\documentclass[11pt,amssymb]{amsart}
\usepackage{amssymb}
\DeclareFontFamily{OT1}{rsfs}{}

\DeclareFontShape{OT1}{rsfs}{n}{it}{<-> rsfs10}{}
\DeclareMathAlphabet{\mathscr}{OT1}{rsfs}{n}{it}

\newcommand{\C}{\mathbb{C}}
\newcommand{\R}{\mathbb{R}}
\newcommand{\Q}{\mathbb{Q}}

\newcommand{\bR}{{\mathbb R}}

\newcommand{\cO}{\mathcal{O}}
\newcommand{\cG}{\mathcal{G}}

\newcommand{\fX}{\mathfrak{X}}

\newcommand{\Ga}{\mathrm{Gal}}
\newtheorem{thm}{Theorem}[section]
\newtheorem{lemma}[thm]{Lemma}
\newtheorem{prop}[thm]{Proposition}
\newtheorem{cor}[thm]{Corollary}

\newcommand{\sG}{\mathscr{G}}
\newcommand{\Z}{\mathbb{Z}}
\newcommand{\sC}{\mathscr{C}}

.240pk scaled 1200 .240pk
\begin{document}
\title[Isospectral locally symmetric spaces]{Generic elements in Zariski-dense subgroups and isospectral locally symmetric spaces}

\author[Prasad]{Gopal Prasad}
\author[Rapinchuk]{Andrei S. Rapinchuk}

\address{Department of Mathematics, University of Michigan, Ann
Arbor, MI 48109}

\email{gprasad@umich.edu}

\address{Department of Mathematics, University of Virginia,
Charlottesville, VA 22904}

\email{asr3x@virginia.edu}

\begin{abstract}
The article contains a survey of our results on length-commensurable and
isospectral locally symmetric spaces and of related problems in the
theory of semi-simple algebraic groups. We discuss some of the techniques
involved in this work (in particular, the existence of generic tori
in semi-simple algebraic groups over finitely generated fields and of generic
elements in finitely generated Zariski-dense subgroups) as well as
some open problems. The article is an expanded version of the talk
given at the workshop by the first-named author.
\end{abstract}

\maketitle

\section{Introduction}\label{S:Intro}
\vskip2mm

The object of the talk delivered at the workshop by the first-named
author was, and of this article is, to give an exposition of recent
results on isospectral and length-commensurable locally symmetric
spaces associated with simple real algebraic groups (\cite{PR-IHES},
\cite{PR-Fields}) and related problems in the theory of semi-simple
algebraic groups (\cite{Gar}, \cite{GarR}, and \cite{PR-invol}). One
of the goals of our paper \cite{PR-IHES} was to study the problem
beautifully formulated by Mark Kac in \cite{Kac} as ``{\it Can one
hear the shape of a drum?}''  for the quotients of symmetric spaces
of the groups of real points of absolutely simple real algebraic
groups by cocompact arithmetic subgroups. A precise mathematical
formulation of Kac's question is {\it whether two compact Riemannian
manifolds which are isospectral} (i.e., have equal spectra --
eigenvalues and multiplicities -- for the Laplace-Beltrami operator)
{\it are necessarily isometric}. In general, the answer to this
question is in the negative as was shown by John Milnor \cite{Milnor}
already in 1964 by constructing two non-isometric isospectral flat
tori of dimension 16. Later M.-F.\,Vigneras \cite{Vig} used
arithmetic properties of quaternion algebras to produce examples of
arithmetically defined isospectral, but not isometric, Riemann
surfaces. On the other hand, T.\,Sunada \cite{Sun}, inspired by a
construction of nonisomorphic number fields with the same Dedekind
zeta-function, proposed a general and basically purely
group-theoretic method of producing nonisometric isospectral
Riemannian manifolds which has since then been used in various ways.
It is important to note, however, that the nonisometric isospectral
manifolds constructed by Vigneras and Sunada are {\it
commensurable}, i.e. have a common finite-sheeted cover. This
suggests that one should probably settle for the following weaker
version of Kac's original question: {\it Are any two isospectral
compact Riemannian manifolds necessarily commensurable?} The answer
to this modified question is still negative in the general case:
Lubotzky, Samuels and Vishne \cite{LSV}, using the Langlands
correspondence, have constructed examples of noncommensurable
isospectral locally symmetric spaces associated with absolutely
simple real groups of type $\textsf{A}_{n}$ (cf.\,Problem
\ref{S:Prob}.7 in \S \ref{S:Prob}). Nevertheless, it turned out that
the answer is actually in the affirmative for several classes of
locally symmetric spaces. Prior to our paper \cite{PR-IHES}, this
was known to be the case only for  the following two classes:
arithmetically defined Riemann surfaces \cite{Reid} and
arithmetically defined hyperbolic 3-manifolds \cite{CHLR}.

\vskip 2mm

In \cite{PR-IHES},  we used Schanuel's conjecture from
transcendental number theory (for more about this conjecture, and
how it comes up in our work, see below) and the results of
\cite{Gar}, \cite{PR-invol} to prove that any two compact {\it
isospectral} arithmetically defined locally symmetric spaces
associated with absolutely simple real algebraic groups of type
other than $\textsf{A}_{n}$ $(n > 1)$, $\textsf{D}_{2n + 1}$ $(n >
1)$, or $\textsf{E}_6$ are {\it necessarily commensurable}. One of
the important ingredients of the proof is the connection between
isospectrality and another property of Riemannian manifolds called
{\it iso-length-spectrality}. More precisely, for a Riemannian
manifold $M$ we let $L(M)$ denote the {\it weak length spectrum} of
$M$, i.e. the collection of the lengths of all closed geodesics in
$M$ (note that for the existence of a ``nice" Laplace spectrum, $M$
is required to be compact, but the weak length spectrum $L(M)$ can
be considered for {\it any} $M$ (i.e., we do not need to assume that
$M$ is compact)). Then two Riemannian manifolds $M_1$ and $M_2$ are
called {\it iso-length spectral} if $L(M_1) = L(M_2)$. It was first
proved by Gangolli \cite{Gang} in the rank one case, and then by
Duistermaat and Guillemin \cite{DG} and Duistermaat, Kolk and
Varadarajan \cite{DKV} in the general case (see Theorem 10.1 in
\cite{PR-IHES}) that  any two compact isospectral locally symmetric
spaces are  iso-length spectral. So, the emphasis in \cite{PR-IHES}
is really on the analysis of iso-length spectral locally symmetric
spaces $M_1$ and $M_2$. In fact, we prove our results under the much
weaker assumption of {\it length-commensurability}, which means that
$\Q \cdot L(M_1) = \Q \cdot L(M_2)$. (The set $\Q \cdot L(M)$ is
sometimes called the {\it rational length spectrum} of $M$; its
advantage, particularly in the analysis of  questions involving
commensurable manifolds, is that it is invariant under passing to a
finite-sheeted cover -- this property fails for the Laplace spectrum
or the length spectrum. At the same time, $\Q \cdot L(M)$ can
actually be computed in at least some cases, while precise
computation of $L(M)$ or the Laplace spectrum is not available for
any compact locally symmetric space at this point.) The notion of
length-commensurability was introduced in \cite{PR-IHES}, and the
investigation of its qualitative and quantitative consequences for
general locally symmetric spaces is an ongoing project. For
arithmetically defined spaces, however, the main questions were
answered in \cite{PR-IHES}, and we would like to complete this
introduction by showcasing the results for arithmetic hyperbolic
spaces.

\vskip2mm

Let $\mathbb{H}^n$ be the real hyperbolic $n$-space. By an
arithmetically defined real hyperbolic $n$-manifold we mean the quotient
$\mathbb{H}^n/\Gamma$, where $\Gamma$ is an arithmetic subgroup of
${\mathrm {PSO}}(n , 1)$ (which is the isometry group of
$\mathbb{H}^n$); see \S \ref{S:WC} regarding the notion of
arithmeticity.
\begin{thm}\label{T:Hyper}
{\rm (cf.\,\cite[Corollary 8.17 and Remark 8.18]{PR-IHES})} Let
$M_1$ and $M_2$ be arithmetically defined real hyperbolic $n$-manifolds.

If $n \not\equiv 1\,(\mathrm{mod}\: 4)$, then in case $M_1$ and
$M_2$ are not commensurable, after a possible interchange of $M_1$
and $M_2$, there exists $\lambda_1 \in L(M_1)$ such that for any
$\lambda_2 \in L(M_2)$, the ratio $\lambda_1/\lambda_2$ is
transcendental over $\Q$. (Thus, for such $n$ the
length-commensurability, and hence isospectrality,  of $M_1$ and
$M_2$ implies their commensurability.)

On the contrary, for any $n \equiv 1(\mathrm{mod}\: 4)$, there exist
$M_1$ and $M_2$ as above that are length-commensurable, but not
commensurable.
\end{thm}

What is noteworthy is that there is no apparent geometric reason for
this dramatic distinction between the length-commensurability of
hyperbolic $n$-manifolds when $n \not\equiv 1(\mathrm{mod}\: 4)$ and
$n \equiv 1(\mathrm{mod}\: 4)$ -- in our argument the difference
comes from considerations involving Galois cohomology -- see Theorem
\ref{T:WC4} and subsequent comments.

Our general results for arithmetically defined length-commensurable
locally symmetric spaces (cf.\,\S \ref{S:WC2}) imply similar (but
not identical!) assertions for complex and quaternionic hyperbolic
manifolds. At the same time, one can ask about possible relations
between $L(M_1)$ and $L(M_2)$ (or between $\Q \cdot L(M_1)$ and $\Q
\cdot L(M_2)$) if $M_1$ and $M_2$ are {\it not}
length-commensurable. The results we will describe in \S
\ref{S:Fields} assert that if $\Q\cdot L(M_1)\ne \Q\cdot L(M_2)$,
then no polynomial-type relation between $L(M_1)$ and $L(M_2)$ can
ever exist; in other words, these sets are {\it very} different.
This is, for example, the case if $M_1$ and $M_2$ are hyperbolic
manifolds of finite volume having different dimensions!

\section{Length-commensurable locally symmetric spaces and weakly commensurable subgroups} \label{S:length}
\vskip2mm

\noindent {\bf 2.1.  Riemann surfaces.} Our analysis of
length-commensurability of locally symmetric spaces relies on a
purely algebraic relation between their fundamental groups which we
termed {\it weak commensurability}. It is  easiest to motivate this
notion by looking at the length-commensurability of Riemann
surfaces. In this discussion we will be using the realization of
$\mathbb{H}^2$ as the complex upper half-plane  with the standard
hyperbolic metric $ds^2 = y^{-2}(dx^2 + dy^2)$. The usual action of
${\mathrm{SL}}_2(\R)$ on $\mathbb{H}^2$ by fractional linear
transformations is isometric and allows us to identify
$\mathbb{H}^2$ with the symmetric space ${\mathrm{SO}}_2(\R)
\backslash {\mathrm{SL}}_2(\R)$. It is well-known that any compact
Riemann surface $M$ of genus $> 1$ can be obtained as a quotient of
$\mathbb{H}^2$ by a discrete subgroup $\Gamma \subset
{\mathrm{SL}}_2(\R)$ with torsion-free image in
${\mathrm{PSL}}_2(\R)$. Now, given  any such subgroup $\Gamma$, we
let $\pi \colon \mathbb{H}^2 \to \mathbb{H}^2/\Gamma =: M$ denote
the canonical projection. It is easy to see that
$$
t \mapsto e^t i = i \cdot \left( \begin{array}{cc} e^{t/2} & 0
\\ 0 & e^{-t/2} \end{array} \right), \ \ \ t \in \R,
$$
is a unit-velocity parametrization of a geodesic $c$ in
$\mathbb{H}^2$. So, if $\gamma = \mathrm{diag}(t_{\gamma} ,
t^{-1}_{\gamma}) \in \Gamma$ then the image $\pi(c)$ is a {\it
closed geodesic} $c_{\gamma}$ in $M$ whose length is given by the
formula
\begin{equation}\label{E:length}
\ell_{\Gamma}(c_{\gamma}) = \frac{2}{n_{\gamma}} \cdot \log
t_{\gamma}
\end{equation}
(assuming that $t_{\gamma} > 1$), where $n_{\gamma}$ is an integer
$\geqslant 1$ (winding number in case $c_{\gamma}$ is not
primitive). Generalizing this construction, one shows that every
semi-simple element $\gamma \in \Gamma \setminus \{ \pm 1 \}$ gives
rise to a closed geodesic $c_{\gamma}$ in $M$ whose length is given
by (\ref{E:length}) where $t_{\gamma}$ is the eigenvalue of $\pm
\gamma$ which is $> 1$, and conversely, any closed geodesic in
$M$ is obtained this way. As a result,
$$
\Q \cdot L(M) = \Q \cdot \{ \log t_{\gamma} \: \vert \: \gamma \in
\Gamma \backslash \{\pm 1 \} \ \ \text{semi-simple} \}.
$$
Now, suppose we have two quotients $M_1 = \mathbb{H}^2/\Gamma_1$ and
$M_2 = \mathbb{H}^2/\Gamma_2$ as above, and let $c_{\gamma_i}$ be a
closed geodesic in $M_i$ for $i = 1, 2$. Then
$$
\ell_{\Gamma_1}(c_{\gamma_1}) / \ell_{\Gamma_2}(c_{\gamma_2}) \: \in
\: \Q \ \ \ \Leftrightarrow \ \ \ \exists \: m , n \in \mathbb{N} \
\ \text{such that} \ \ t^m_{\gamma_1} = t^n_{\gamma_2},
$$
or equivalently, the subgroups generated by the eigenvalues of
$\gamma_1$ and $\gamma_2$ have nontrivial intersection. This leads
us to the following.

\vskip2mm

\noindent {\bf 2.2. Definition.} Let $G_1 \subset \mathrm{GL}_{N_1}$ and
$G_2 \subset \mathrm{GL}_{N_2}$ be two semi-simple algebraic groups
defined over a field $F$ of characteristic zero.

\vskip1mm

\noindent (a) Semi-simple elements $\gamma_1 \in G_1(F)$ and
$\gamma_2 \in G_2(F)$ are said to be {\it weakly commensurable} if
the subgroups of $\overline{F}^{\times}$ generated by their
eigenvalues intersect nontrivially.

\vskip1mm

\noindent (b) (Zariski-dense) subgroups $\Gamma_1 \subset G_1(F)$
and $\Gamma_2 \subset G_2(F)$ are {\it weakly commensurable} if
every semi-simple element $\gamma_1 \in \Gamma_1$ of infinite order
is weakly commensurable to some semi-simple element $\gamma_2 \in
\Gamma_2$ of infinite order, and vice versa.

\vskip2mm

It should be noted that in \cite{PR-IHES} we gave a more technical,
but equivalent, definition of weakly commensurable elements, viz. we
required that there should exist maximal $F$-tori $T_i$ of $G_i$ for
$i = 1, 2$ such that $\gamma_i \in T_i(F)$ and for some characters
$\chi_i \in X(T_i)$ we have
$$
\chi_1(\gamma_1) = \chi_2(\gamma_2) \neq 1.
$$
This (equivalent) reformulation of (a) immediately demonstrates that
the notion of weak commensurability does not depend on the choice of
matrix realizations of the $G_i$'s, and more importantly, is more
convenient for the proofs of our results.

The above discussion of Riemann surfaces implies that if two Riemann
surfaces $M_1 = \mathbb{H}^2/\Gamma_1$ and $M_2 =
\mathbb{H}^2/\Gamma_2$ are length-commensurable, then the
corresponding fundamental groups $\Gamma_1$ and $\Gamma_2$ are
weakly commensurable. Our next goal is to explain why this
implication remains valid for general locally symmetric spaces.

\vskip2mm

\noindent {\bf 2.3. Length-commensurability and weak commensurability:\:the general case.}
\vskip2mm

First, we need to fix some notations related to general locally
symmetric spaces. Let $G$ be a connected adjoint real semi-simple
algebraic group, let $\mathcal{G} = G(\R)$ considered as a real Lie
group, and let $\fX = \mathcal{K} \backslash \mathcal{G}$, where
$\mathcal{K}$ is a maximal compact subgroup of $\mathcal{G}$, be the
associated symmetric space endowed with the Riemannian metric coming
from the Killing form on the Lie algebra $\mathfrak{g}$ of
$\mathcal{G}$. Furthermore, given a torsion-free discrete subgroup
$\Gamma$ of $\mathcal{G}$, we let $\fX_{\Gamma} = \fX/\Gamma$ denote
the corresponding locally symmetric space. Just as in the case of
Riemann surfaces, to any nontrivial semi-simple element $\gamma \in
\Gamma$ there corresponds a closed geodesic $c_{\gamma}$ whose
length is given by
$$
\ell_{\Gamma}(c_{\gamma}) = \frac{1}{n_{\gamma}} \cdot
\lambda_{\Gamma}(\gamma),
$$
where $n_{\gamma}$ is an integer $\geqslant 1$ and
\begin{equation}\label{E:lambda}
\lambda_{\Gamma}(\gamma)^2 := \sum_{\alpha} (\log \vert
\alpha(\gamma) \vert )^2,
\end{equation}
with the summation running over all roots of $G$ with respect to a
fixed maximal $\R$-torus $T$ of $G$ whose group of $\R$-points contains
 $\gamma$ ($\ell_{\Gamma}(c_{\gamma})$ is thus a {\it submultiple} of
$\lambda_{\Gamma}(\gamma)$). This formula looks much more
intimidating than (\ref{E:length}), so in order to make it more
manageable we first make the following observation. Of course, the
$\R$-torus $T$ may not be $\R$-split, so not every root $\alpha$ may
be defined over $\R$. However,
$$
\vert \alpha(\gamma) \vert^2 = \chi(\gamma)
$$
where $\chi = \alpha + \overline{\alpha}$ (with $\overline{\alpha}$ being the
conjugate character in terms of the natural action of $\Ga(\C/\R)$
on $X(T)$, and, as usual, $X(T)$ is viewed as an additive group) is a  character defined over $\R$
and which takes positive values on $T(\R)$. Such
characters will be called {\it positive}. So, we can now re-write
(\ref{E:lambda}) in the form
\begin{equation}\label{E:positive}
\lambda_{\Gamma}(\gamma)^2 = \sum_{i = 1}^p s_i (\log
\chi_i(\gamma))^2
\end{equation}
where $\chi_1, \ldots , \chi_p$ are certain positive characters of
$T$ and $s_1, \ldots , s_p$ are positive rational numbers whose
denominators are divisors of 4. The point to be made here is that
the subgroup $P(T) \subset X(T)$ of positive characters may be
rather small. More precisely, $T$ is an almost direct product of an
$\R$-anisotropic subtorus $A$ and an $\R$-split subtorus $S$. Then
any character of $T$ which is defined over $\R$ vanishes on $A$. This easily implies
that the restriction map yields an embedding $P(T) \hookrightarrow
X(S)_{\R} = X(S)$ with finite cokernel; in particular, the rank of
$P(T)$ as an abelian group coincides with the $\R$-rank
$\mathrm{rk}_{\R}\: T$ of $T$.

Before formulating our results,  we define the following property.
Let $G \subset \mathrm{GL}_N$ be a semi-simple algebraic group
defined over a field $F$ of characteristic zero. We say that a
(Zariski-dense) subgroup $\Gamma \subset G(F)$ has {\it property
$(A)$} if for any semi-simple element $\gamma \in \Gamma$, all the
eigenvalues of $\gamma$ lie in the field of algebraic numbers
$\overline{\Q}$ (note that the latter is equivalent to the fact that
for any maximal $F$-torus $T$ of $G$ containing $\gamma$ and any
character $\chi \in X(T)$, we have $\chi(\gamma) \in
\overline{\Q}^{\times}$ -- this reformulation shows, in particular,
that this property does not depend on the choice of a matrix
realization of $G$). Of course, this property automatically holds if
$\Gamma$ is arithmetic, or more generally, if $\Gamma$ can be
conjugated into ${\mathrm{SL}}_N(K)$ for some number field $K$.

\vskip2mm
Let us now consider the rank one case first.
\vskip1mm

{\it The rank one case.} Suppose $\mathrm{rk}_{\R}\: G = 1$ (the
examples include the adjoint groups of $\mathrm{SO}(n , 1)$,
$\mathrm{SU}(n , 1)$ and $\mathrm{Sp}(n , 1)$; the corresponding
symmetric spaces are respectively the real, complex and quaternionic
hyperbolic $n$-spaces). Then given a nontrivial semi-simple element $\gamma \in
\Gamma$, for any maximal $\R$-torus $T$ of $G$ containing $\gamma$
we have $\mathrm{rk}_{\R}\: T = 1$, which implies that the group
$P(T)$ of positive characters is cyclic and is generated, say, by
$\chi$. Then it follows from (\ref{E:positive}) that
\begin{equation}\label{E:lambda2}
\lambda_{\Gamma}(\gamma) = \frac{\sqrt{m}}{2} \cdot \vert \log
\chi(t) \vert
\end{equation}
where $m$ is some integer $\geqslant 1$ depending only on $G$; note
that this formula is still in the spirit of (\ref{E:length}), but
potentially involves some irrationality which can complicate the
analysis of length-commensurability.

Now, suppose that $G_1$ and $G_2$ are two simple algebraic
$\R$-groups of $\R$-rank one. For $i = 1, 2$, let $\Gamma_i \subset
G_i(\R) = \mathcal{G}_i$ be a discrete torsion-free subgroup
having property (A). Given a nontrivial semi-simple element
$\gamma_i \in \Gamma_i$, we pick a maximal $\R$-torus $T_i$ of $G_i$
whose group of $\R$-points contains $\gamma_i$ and let $\chi_i$ be a generator of the group
of positive characters $P(T_i)$. Then according to
(\ref{E:lambda2}),
$$
\lambda_{\Gamma_1}(\gamma_1) = \frac{\sqrt{m_1}}{2} \cdot \vert \log
\chi_1(\gamma_1) \vert \ \ \ \text{and} \ \ \
\lambda_{\Gamma_2}(\gamma_2) = \frac{\sqrt{m_2}}{2} \cdot \vert \log
\chi_2(\gamma_2) \vert
$$
for some integers $m_1 , m_2 \geqslant 1$. By a theorem proved
independently by Gel'fond and Schneider in 1934 (which settled
Hilbert's seventh problem -- cf.\,\cite{Baker}), the ratio
$$
\frac{\log \chi_1(\gamma_1)}{\log \chi_2(\gamma_2)}
$$
is either rational or transcendental. This result implies that the
ratio $\ell_{\Gamma_1}(\gamma_1)/\ell_{\Gamma_2}(\gamma_2)$, or
equivalently, the ratio
$\lambda_{\Gamma_1}(\gamma_1)/\lambda_{\Gamma_2}(\gamma_2)$ can be
rational only if
$$
\chi_1(\gamma_1)^{n_1} = \chi_2(\gamma_2)^{n_2}
$$
for some nonzero integers $n_1 , n_2$, which makes the elements
$\gamma_1$ and $\gamma_2$ weakly commensurable. (Of course, we get
this conclusion without  using the theorem of Gel'fond-Schneider if $G_1 =
G_2$, hence $m_1 = m_2$.) This argument shows that the
length-commensurability of $\fX_{\Gamma_1}$ and $\fX_{\Gamma_2}$
implies the weak commensurability of $\Gamma_1$ and $\Gamma_2$.

Finally, we recall that if $G \subset \mathrm{GL}_N$ is an
absolutely simple real algebraic group not isomorphic to
$\mathrm{PGL}_2$ then any lattice $\Gamma \subset G(\R)$ can be
conjugated into ${\mathrm{SL}}_N(K)$ for some number field $K$
(cf.\,\cite[7.67 and 7.68]{Rag-book}), hence possesses property (A).
This implies that if $\fX_{\Gamma_1}$ and $\fX_{\Gamma_2}$ are rank
one locally symmetric spaces of finite volume then their
length-commensurability always implies the weak commensurability of
$\Gamma_1$ and $\Gamma_2$ except possibly in the following
situation: $G_1 = \mathrm{PGL}_2$ and $\Gamma_1$ cannot be
conjugated into $\mathrm{PGL}_2(K)$ for any number field $K \subset
\R$ while $G_2 \neq \mathrm{PGL}_2$ (in
\cite{PR-IHES} this was called the exceptional case $(\mathcal{E})$). Nevertheless,
the conclusion remains valid also in this case if one assumes the
truth of Schanuel's conjecture (see below) -- this
follows from our recent results \cite{PR-Fields}, which we will
discuss in \S \ref{S:Fields} (cf.\,Theorem \ref{T:Fields1}).

\vskip2mm

{\it The general case}. If $\mathrm{rk}_{\R}\: G > 1$ then  $p$ may
be $> 1$ in (\ref{E:positive}), hence $\lambda_{\Gamma}(\gamma)$,
generally speaking, is {\it not} a multiple of the logarithm of the
value of a positive character. Consequently, the fact that the
ratio $\lambda_{\Gamma_1}(\gamma_1)/\lambda_{\Gamma_2}(\gamma_2)$ is
a rational number does not imply {\it directly} that $\gamma_1$ and
$\gamma_2$ are weakly commensurable. While the implication
nevertheless {\it is} valid (under some natural technical
assumptions), it is hardly surprising now that the proof requires
some nontrivial information about the logarithms of the character
values. More precisely, our arguments in \cite{PR-IHES} and
\cite{PR-Fields} assume the truth of the following famous conjecture
in transcendental number theory (cf.\,\cite{Ax}).

\vskip2mm

\noindent {\bf 2.4. Schanuel's conjecture.} {\it If $z_1, \ldots , z_n
\in \C$ are linearly independent over $\Q$, then the transcendence
degree (over $\Q$) of the field generated by
$$
z_1, \ldots , z_n; \ e^{z_1}, \ldots , e^{z_n}
$$
is $\geqslant n$.}

\vskip2mm

In fact, we will only need the consequence of this conjecture that
for nonzero algebraic numbers $z_1, \ldots , z_n$, (any values of)
their logarithms $\log z_1, \ldots , \log z_n$ are algebraically
independent once they are linearly independent (over $\Q$). In order
to apply this statement in our situation, we first prove the
following elementary lemma.

\addtocounter{thm}{4}

\begin{lemma}\label{L:1}
Let $G_1$ and $G_2$ be two connected semi-simple real algebraic
groups. For $i = 1, 2$, let $T_i$ be a maximal $\R$-torus of $G_i$,
$\gamma_i \in T_i(\R)$ and let $\chi^{(i)}_1, \ldots ,
\chi^{(i)}_{d_i}$ be positive characters of $T_i$ such that the set
$$
S_i = \{ \log \chi^{(i)}_1(\gamma_i), \ldots , \log
\chi^{(i)}_{d_i}(\gamma_i) \} \subset \R
$$
is linearly independent over $\Q$. If $\gamma_1$ and $\gamma_2$ are
\emph{not} weakly commensurable then the set $S_1 \cup S_2$ is also
linearly independent.
\end{lemma}
\begin{proof}
Assume the contrary. Then there exist integers $s_1, \ldots ,
s_{d_1}, t_1, \ldots , t_{d_2}$, not all zero, such that
$$
s_1 \log \chi^{(1)}_1(\gamma_1) + \cdots + s_{d_1} \log
\chi^{(1)}_{d_1}(\gamma_1) - t_1 \log \chi^{(2)}_1(\gamma_2) -
\cdots - t_{d_2} \log \chi^{(2)}_{d_2}(\gamma_2) = 0.
$$
Consider the following characters
$$
\psi^{(1)} := s_1\chi^{(1)}_1 + \cdots + s_{d_1}\chi^{(1)}_{d_1} \ \
\text{and} \ \ \psi^{(2)} := t_1 \chi^{(2)}_1 + \cdots + t_{d_2}
\chi^{(2)}_{d_2}
$$
of $T_1$ and $T_2$ respectively. Then $\psi^{(1)}(\gamma_1) =
\psi^{(2)}(\gamma_2)$, and hence
$$
\psi^{(1)}(\gamma_1) = 1 = \psi^{(2)}(\gamma_2)
$$
because $\gamma_1$ and $\gamma_2$ are not weakly commensurable. This
means that
$$
s_1 \log \chi^{(1)}_1(\gamma_1) + \cdots + s_{d_1} \log
\chi^{(1)}_{d_1}(\gamma_1) = 0 =  t_1 \log \chi^{(2)}_1(\gamma_2) +
\cdots + t_{d_2} \log \chi^{(2)}_{d_2}(\gamma_2),
$$
and therefore all the coefficients are zero because the sets $S_1$
and $S_2$ are linearly independent. A contradiction.
\end{proof}

We are now ready to connect length-commensurability with weak
commensurability.
\begin{prop}\label{P:LC-WC}
Let $G_1$ and $G_2$ be two connected semi-simple real algebraic
groups. For $i = 1, 2$, let $\Gamma_i \subset G_i(\R)$ be a subgroup
satisfying property (A). Assume that Schanuel's conjecture holds. If
semi-simple elements $\gamma_1 \in \Gamma_1$ and $\gamma_2 \in
\Gamma_2$ are not weakly commensurable then
$\lambda_{\Gamma_1}(\gamma_1)$ and $\lambda_{\Gamma_2}(\gamma_2)$
are algebraically independent over $\Q$.
\end{prop}
\begin{proof}
It follows from (\ref{E:positive})  that
\begin{equation}\label{E:lambda4}
\lambda_{\Gamma_1}(\gamma_1)^2 = \sum_{i = 1}^p s_i(\log
\chi^{(1)}_i(\gamma_1))^2 \ \ \text{and} \ \
\lambda_{\Gamma_2}(\gamma_2)^2 = \sum_{i = 1}^q t_i (\log
\chi^{(2)}_i(\gamma_2))^2,
\end{equation}
where $s_i$ and $t_i$ are positive rational numbers, and
$\chi^{(1)}_i$ and $\chi^{(2)}_i$ are positive characters on maximal
$\R$-tori $T_1$ and $T_2$ of $G_1$ and $G_2$ whose groups of $\R$-points contain the
elements $\gamma_1$ and $\gamma_2$, respectively. After renumbering
the characters, we can assume that
$$
a_1 := \log \chi^{(1)}_1(\gamma_1), \ldots , a_{m} := \log
\chi^{(1)}_{m}(\gamma_1)
$$
(resp., $b_1 := \log \chi^{(2)}_1(\gamma_2), \ldots , b_n = \log
\chi^{(2)}_n(\gamma_2)$) form a basis of the $\Q$-subspace of $\R$
spanned by $\log \chi^{(1)}_i(\gamma_1)$ for $i \leqslant p$ (resp.,
by $\log \chi^{(2)}_i(\gamma_2)$ for $i \leqslant q$). It follows
from Lemma \ref{L:1} that the numbers
\begin{equation}\label{E:numbers}
a_1, \ldots , a_{m}; \ b_1, \ldots , b_n
\end{equation}
are linearly independent. By our assumption, $\Gamma_1$ and
$\Gamma_2$ possess property (A), so the character values
$\chi^{(j)}_i(\gamma_j)$ are all algebraic numbers. So, it follows
from Schanuel's conjecture that the numbers in (\ref{E:numbers}) are
algebraically independent over $\Q$. As is seen from (\ref{E:lambda4}),
$\lambda_{\Gamma_1}(\gamma_1)$ and
$\lambda_{\Gamma_2}(\gamma_2)$ are represented by nonzero
homogeneous polynomials of degree two, with rational coefficients,
in $a_1, \ldots , a_{m}$ and $b_1, \ldots , b_n$, respectively,
and therefore they are algebraically independent.
\end{proof}

This  proposition leads us to the following.

\begin{thm}\label{T:LC-WC}
Let $G_1$ and $G_2$ be two connected semi-simple real algebraic
groups. For $i = 1, 2$, let $\Gamma_i \subset G_i(\R)$ be a discrete
torsion-free subgroup having property (A). Assume that Schanuel's
conjecture holds. If $\Gamma_1$ and $\Gamma_2$ are not weakly
commensurable, then, possibly after reindexing, we can find $\lambda_1 \in L(\fX_{\Gamma_1})$ which is
algebraically independent from \emph{any} $\lambda_2 \in
L(\fX_{\Gamma_2})$. In particular, $\fX_{\Gamma_1}$ and
$\fX_{\Gamma_2}$ are not length-commensurable.
\end{thm}

Combining this with the discussion above of property (A) for lattices and of
the exceptional case $(\mathcal{E})$, we obtain the following.

\begin{cor}\label{C:LC-WC}
Let $G_1$ and $G_2$ be two absolutely simple real algebraic groups,
and for $i = 1, 2$ let $\Gamma_i$ be a lattice in $G_i(\R)$ (so that
the locally symmetric space $\fX_{\Gamma_i}$ has finite volume). If
$\fX_{\Gamma_1}$ and $\fX_{\Gamma_2}$ are length-commensurable then
$\Gamma_1$ and $\Gamma_2$ are weakly commensurable.
\end{cor}

\vskip1mm

The results we discussed in this section shift the focus in the
analysis of length-commensurability and/or isospectrality of locally
symmetric spaces to that of weak commensurability of finitely
generated Zariski-dense subgroups of simple (or semi-simple)
algebraic  groups. In \S\ref{S:WC}, we will first present some basic
results dealing with the weak commensurability of such subgroups in
a completely general situation, one of which states that the mere
existence of such subgroups implies that the ambient algebraic
groups either are of the same type, or one of them is of type
$\textsf{B}_{n}$ and the other of type $\textsf{C}_{n}$ for some $n
\geqslant 3$ (cf.\,Theorem \ref{T:WC1}). We then turn to much more
precise results in the case where the algebraic groups are of the
same type and the subgroups are $S$-arithmetic (see \S
\ref{S:WC-Res}), and finally derive some geometric consequences of
these results (see \S \ref{S:WC2}). Next, \S\ref{S:BC} contains an
exposition of the recent results of Skip Garibaldi and the
second-named author \cite{GarR} that completely characterize weakly
commensurable $S$-arithmetic subgroups in the case where one of the
two groups is of type $\textsf{B}_{n}$ and the other is of type
$\textsf{C}_{n}$ $(n \geqslant 3)$. In \S \ref{S:Fields} we discuss
a more technical version of the notion of weak commensurability,
which enabled us to show in \cite{PR-Fields} (under  mild technical
assumptions) that if two arithmetically defined locally symmetric
spaces $M_1 = \fX_{\Gamma_1}$ and $M_2 = \fX_{\Gamma_2}$ are not
length-commensurable then the sets $L(M_1)$ and $L(M_2)$ (or $\Q
\cdot L(M_1)$ and $\Q \cdot (M_2)$) are {\it very} different. The
proofs of all these results  use the existence (first established in
\cite{PR-Reg}) of special elements, which we call {\it generic
elements}, in arbitrary finitely generated Zariski-dense subgroups;
we briefly review these and more recent results in this direction in
\S\ref{S:Gen} along with the results that relate the analysis of
weak commensurability with a problem of independent interest in the
theory of semi-simple algebraic groups of characterizing simple
$K$-groups having the same isomorphism classes of maximal $K$-tori
(cf. \S \ref{S:Tori}). Finally, in \S\ref{S:Prob} we discuss some
open problems.

\vskip2mm

\section{Two basic results implied by weak commensurability and the definition of arithmeticity}\label{S:WC}

\noindent Our next goal is to give an
account of the results from \cite{PR-IHES} concerning weakly
commensurable subgroups of semi-simple algebraic groups. We begin with the
following two theorems that provide the basic results about weak
commensurability of arbitrary finitely generated Zariski-dense
subgroups of semi-simple groups.
\begin{thm}\label{T:WC1}
Let $G_1$ and $G_2$ be two connected absolutely almost simple
algebraic groups defined over a field $F$ of characteristic zero.
Assume that there exist finitely generated Zariski-dense subgroups
$\Gamma_i$ of  $G_i(F)$ which are weakly commensurable. Then either
$G_1$ and $G_2$ are of the same Killing-Cartan type, or one of them
is of type $\textsf{B}_{n}$ and the other is of type
$\textsf{C}_{n}$ for some $n \geqslant 3$.
\end{thm}

\vskip1mm

The way we prove this theorem is by showing that the Weyl groups of
$G_1$ and $G_2$ have the same order, as it is well-known that the
order of the Weyl group uniquely determines the type of the root
system, except for the ambiguity between $\textsf{B}_{n}$ and
$\textsf{C}_{n}$. On the other hand, groups $G_1$ and $G_2$ of types
$\textsf{B}_{n}$ and $\textsf{C}_{n}$ with $n > 2$ respectively, may
indeed contain weakly commensurable subgroups. This was first shown
in \cite[Example 6.7]{PR-IHES} using a cohomological construction
which we will briefly recall in \S \ref{S:BC}. Recently in
\cite{GarR} another explanation was given using commutative  \'etale
subalgebras of simple algebras with involution. We refer the reader
to \S \ref{S:BC} for this argument as well as a complete
characterization of weakly commensurable $S$-arithmetic subgroups in
the algebraic groups of types $\textsf{B}_{n}$ and $\textsf{C}_{n}$
(see Theorem \ref{T:BC1}).

\vskip1mm

\begin{thm}\label{T:WC2}
Let $G_1$ and $G_2$ be two connected absolutely almost simple
algebraic groups defined over a field $F$ of characteristic zero.
For $i =1,2$, let $\Gamma_i$ be a finitely generated Zariski-dense
subgroup of $G_i(F)$, and $K_{\Gamma_i}$ be the subfield of $F$
generated by the traces ${\mathrm{Tr}}\, \mathrm{Ad}\:\gamma$, in the
adjoint representation,  of  $\gamma \in \Gamma_i$. If $\Gamma_1$
and $\Gamma_2$ are weakly commensurable, then $K_{\Gamma_1} =
K_{\Gamma_2}.$
\end{thm}

\vskip2mm

We now turn to the results concerning weakly commensurable
Zariski-dense $S$-arithmetic subgroups, which are surprisingly
strong. In \S \ref{S:WC-Res} we will discuss the weak
commensurability of $S$-arithmetic subgroups in absolutely almost
simple algebraic groups $G_1$ and $G_2$ of the same type, postponing
the case where one of the groups is of type $\textsf{B}_{n}$ and the
other of type $\textsf{C}_{n}$ to \S \ref{S:BC}. Since our results
rely on a specific way of describing $S$-arithmetic subgroups  in
absolutely almost simple groups, we will discuss this issue first.

\vskip2mm

\noindent{\bf  3.3. The definition of arithmeticity:} Let $G$ be an
algebraic group defined over a number field $K$, and let $S$ be a
finite subset of the set $V^K$ of all places of $K$ containing the
set $V_{\infty}^K$ of archimedean places. Fix a
$K$-embedding $G \subset \mathrm{GL}_N$, and consider the group of
$S$-integral points
$$
G(\cO_K(S)) := G \cap {\mathrm{GL}}_N(\cO_K(S)).
$$
Then, for any field extension $F/K$, the subgroups of $G(F)$ that
are commensurable\footnote{We recall that two subgroups
$\mathcal{H}_1$ and $\mathcal{H}_2$ of an abstract group
$\mathcal{G}$ are called {\it commensurable} if their intersection
$\mathcal{H}_1 \cap \mathcal{H}_2$ is of finite index in each of the
subgroups.} with $G(\cO_K(S))$ are called $S$-{\it arithmetic}, and
in the case where $S = V_{\infty}^K$ simply {\it arithmetic} (note
that $\cO_K(V_{\infty}^K) = \cO_K$, the ring of algebraic integers
in $K$). It is well-known  that the resulting class of
$S$-arithmetic subgroups does not depend on the choice of
$K$-embedding $G \subset \mathrm{GL}_N$ (cf.\,\cite{PlR}). The
question, however, is what we should mean by an arithmetic subgroup
of $G(F)$ when $G$ is an algebraic group defined over a field $F$ of
characteristic zero that is not equipped with a structure of
$K$-group over some number field $K \subset F$. For example, what is
an arithmetic subgroup of $G(\R)$ where $G = \mathrm{SO}_3(f)$ and
$f = x^2 + e y^2 - \pi z^2$? For absolutely almost simple groups the
``right" concept that we will formalize below is given in terms of
the forms of $G$ over the subfields $K \subset F$ that are number
fields. In our example, we can consider the following rational
quadratic forms that are equivalent to $f$ over $\R$:
$$
f_1 = x^2 + y^2 - 3z^2 \ \ \ \text{and} \ \ \ f_2 = x^2 + 2y^2 -
7z^2,
$$
and set $G_i = \mathrm{SO}_3(f_i)$. Then for each $i = 1, 2$, we
have an $\R$-isomorphism $G_i \simeq G$, so the natural arithmetic
subgroup  $G_i(\Z) \subset G_i(\R)$ can be thought of as an
``arithmetic" subgroup of $G(\R)$. Furthermore, one can consider
quadratic forms over other number subfields $K \subset \R$ that
again become equivalent to $f$ over $\R$; for example,
$$
K = \Q(\sqrt{2}) \ \ \ \text{and} \ \ \ f_3 = x^2 + y^2 - \sqrt{2}
z^2.
$$
Then for $G_3 = \mathrm{SO}_3(f_3)$, there is an $\R$-isomorphism
$G_3 \simeq G$ which allows us to view the natural arithmetic
subgroup $G_3(\cO_K) \subset G_3(\R)$, where $\cO_K = \Z[\sqrt{2}]$,
as an ``arithmetic" subgroup of $G(\R)$. One can easily generalize
such constructions from arithmetic to $S$-arithmetic groups by
replacing the rings of integers with the rings of $S$-integers. So,
generally speaking, by an $S$-arithmetic subgroup of $G(\R)$ we mean
a subgroup which is commensurable to one of the subgroups obtained
through this construction for some choice of a number subfield $K
\subset \R$, a finite set $S$ of places of $K$ containing all the
archimedean ones, and a quadratic form $\tilde{f}$  over $K$ that
becomes equivalent to $f$ over $\R$. The technical definition is as
follows.

\vskip2mm

Let $G$ be a connected  absolutely almost simple  algebraic group
defined over a field $F$ of characteristic zero, $\overline{G}$ be
its adjoint group, and $\pi \colon G \to \overline{G}$ be the
natural isogeny. Suppose we are given the following data:

\vskip2mm

$\bullet$ \parbox[t]{13.5cm}{a {\it number field} $K$ with a {\it
fixed} embedding $K \hookrightarrow F$;}

\vskip1mm

$\bullet$ \parbox[t]{13.5cm}{an $F/K$-form $\mathscr{G}$ of
$\overline{G}$, which is an algebraic $K$-group such that there
exists an $F$-isomorphism ${}_F \mathscr{G} \simeq \overline{G}$,
where ${}_F\mathscr{G}$ is the group obtained from $\mathscr{G}$ by
the extension of scalars from $K$ to $F$;}

\vskip1mm

$\bullet$ \parbox[t]{13.5cm}{a finite set $S$ of places of $K$
containing $V_{\infty}^K$ but not containing any nonarchimedean
places $v$ such that $\mathscr{G}$ is
$K_v$-anisotropic\footnotemark.}

\footnotetext{We note that if $\mathscr{G}$ is $K_v$-anisotropic
then $\mathscr{G}(\cO_K(S))$ and $\mathscr{G}(\cO_K(S \cup \{v \}))$
are commensurable, and therefore the classes of $S$- and $(S \cup \{
v \})$-arithmetic subgroups coincide. Thus, this assumption on $S$
is necessary if we want to recover it from a given $S$-arithmetic
subgroup.}

\vskip2mm

\noindent We then have an embedding $\iota \colon \mathscr{G}(K)
\hookrightarrow \overline{G}(F)$ which is well-defined up to an
$F$-automorphism of $\overline{G}$ (note that we do {\it not} fix
an isomorphism ${}_F\mathscr{G} \simeq \overline{G}$). A subgroup
$\Gamma$ of $G(F)$ such that $\pi(\Gamma)$ is commensurable with
$\sigma(\iota (\mathscr{G}(\cO_{K}(S))))$, for some $F$-automorphism
$\sigma$ of $\overline{G}$, will be called a $(\mathscr{G} , K ,
S)$-{\it arithmetic subgroup}\footnote{This notion of arithmetic
subgroups coincides with that in Margulis' book \cite{Mar} for
absolutely simple adjoint groups.}, or an $S$-arithmetic subgroup described in terms of the triple
$(\mathscr{G}, K,S)$. As usual, $(\mathscr{G} , K,
V^{K}_{\infty})$-arithmetic subgroups will simply be called
$(\mathscr{G} , K)$-arithmetic.

\vskip2mm

We also need to introduce a more general notion of commensurability.
The point is that since weak commensurability is defined in terms of
eigenvalues, a subgroup $\Gamma \subset G(F)$ is weakly
commensurable with any conjugate subgroup, while the latter may not
be commensurable with the former in the usual sense. So, to make
theorems asserting that in certain situations ``{\it weak
commensurability implies commensurability}" possible (and such
theorems are in fact one of the goals of our analysis) one
definitely needs to modify the notion of commensurability. The
following notion works well in geometric applications. Let $G_i,$
for $i = 1,\,2,$ be a connected absolutely almost simple $F$-group,
and let $\pi_i \colon G_i \to \overline{G}_i$ be the isogeny onto
the corresponding adjoint group. We will say that the subgroups
$\Gamma_i$ of $G_i(F)$ are {\it commensurable up to an
$F$-isomorphism between $\overline{G}_1$ and $\overline{G}_2$} if
there exists an $F$-isomorphism $\sigma \colon \overline{G}_1 \to
\overline{G}_2$ such that $\sigma(\pi_1(\Gamma_1))$ is commensurable
with $\pi(\Gamma_2)$ in the usual sense. The key observation is that
the description of $S$-arithmetic subgroups in terms of triples
$(\mathscr{G},K,S)$ is very convenient for determining when two such
subgroups are commensurable in the new generalized sense.

\addtocounter{thm}{1}
\begin{prop}\label{P:WC1}
Let $G_1$ and $G_2$ be connected absolutely almost simple algebraic
groups defined over a field $F$ of characteristic zero, and for $i =
1, 2$, let $\Gamma_i$ be a Zariski-dense $(\mathscr{G}_i, K_i,
S_i)$-arithmetic subgroup of $G_i(F)$. Then $\Gamma_1$ and
$\Gamma_2$ are commensurable up to an $F$-isomorphism between
$\overline{G}_1$ and $\overline{G}_2$ if and only if $K_1 = K_2 =:
K$, $S_1 = S_2$, and $\mathscr{G}_1$ and $\mathscr{G}_2$ are
$K$-isomorphic.
\end{prop}

It follows from the above proposition that the arithmetic subgroups $\Gamma_1$,
$\Gamma_2$, and $\Gamma_3$ constructed above, of $G(\R)$, where $G =
\mathrm{SO}_3(f)$, are pairwise noncommensurable: indeed,
$\Gamma_3$, being defined over $\Q(\sqrt{2})$, cannot possibly be
commensurable to $\Gamma_1$ or $\Gamma_2$ as these two groups are
defined over $\Q$; at the same time, the non-commensurability of
$\Gamma_1$ and $\Gamma_2$ is a consequence of the fact that
$\mathrm{SO}_3(f_1)$ and $\mathrm{SO}_3(f_2)$ are not
$\Q$-isomorphic since the quadratic forms $f_1$ and $f_2$ are not equivalent over $\Q$.

\vskip2mm

\section{Results on weakly commensurable $S$-arithmetic
subgroups}\label{S:WC-Res}

\vskip2mm

In view of Proposition \ref{P:WC1}, the central question in the
analysis of weak commensurability of $S$-arithmetic subgroups is the
following: {\it Suppose we are given two Zariski-dense
$S$-arithmetic subgroups that are described in terms of triples.
Which components of these triples coincide given the fact that the
subgroups are weakly commensurable?} As the following result
demonstrates, two of these components {\it must} coincide.
\begin{thm}\label{T:WC3}
Let $G_1$ and $G_2$ be two connected absolutely almost simple
algebraic groups defined over a field $F$ of characteristic zero. If
Zariski-dense $(\mathscr{G}_i , K_i , S_i)$-arithmetic subgroups
$\Gamma_i$ of $G_i(F),$ where $i = 1,\,2,$ are weakly commensurable
for $i = 1 , 2,$ then $K_1 = K_2$ and $S_1 = S_2.$
\end{thm}

\vskip1mm

In general, the forms $\mathscr{G}_1$ and $\mathscr{G}_2$ do not
have to be $K$-isomorphic (see \cite{PR-IHES}, Examples 6.5 and 6.6
as well as the general construction in \S 9). In the next theorem we
list the cases where it can nevertheless be asserted that
$\mathscr{G}_1$ and $\mathscr{G}_2$ are necessarily $K$-isomorphic,
and then give a general finiteness result for the number of
$K$-isomorphism classes.
\begin{thm}\label{T:WC4}
Let $G_1$ and $G_2$ be two connected absolutely almost simple
algebraic groups defined over a field $F$ of characteristic zero, of
the \emph{same} type different from $\textsf{A}_{n}$,
$\textsf{D}_{2n+1}$, with $n > 1$, or $\textsf{E}_6$. If for
$i = 1,\,2 $, $G_i(F)$  contain Zariski-dense weakly commensurable
$(\mathscr{G}_i , K , S)$-arithmetic subgroups $\Gamma_i,$ then
$\mathscr{G}_1 \simeq \mathscr{G}_2$ over $K,$ and hence $\Gamma_1$
and $\Gamma_2$ are commensurable up to an $F$-isomorphism between
$\overline{G}_1$ and $\overline{G}_2.$
\end{thm}

In this theorem, type $\textsf{D}_{2n}$ $(n \geqslant 2)$ required
special consideration. The case $n > 2$ was settled in
\cite{PR-invol} using the techniques of \cite{PR-IHES} in
conjunction with  results on embeddings of fields with involutive
automorphisms into simple algebras with involution. The remaining
case of type $\textsf{D}_4$ was treated by Skip Garibaldi
\cite{Gar}, whose argument actually applies to all $n$ and explains
the result from the perspective of Galois cohomology, providing
thereby a cohomological insight into the difference between the
types $\textsf{D}_{2n}$ and $\textsf{D}_{2n + 1}$. We note that the
types excluded in the theorem are precisely the types for which the
automorphism $\alpha \mapsto -\alpha$ of the corresponding root
system is not in the Weyl group. More importantly, all these types
are honest exceptions to the theorem -- a general
Galois-cohomological construction of weakly commensurable, but not
commensurable, Zariski-dense $S$-arithmetic subgroups for all of
these types is given in \cite[\S 9]{PR-IHES}.

\vskip2mm

\begin{thm}\label{T:WC5}
Let $G_1$ and $G_2$ be two connected absolutely almost simple groups
defined over a field $F$ of characteristic zero. Let $\Gamma_1$ be a
Zariski-dense $(\mathscr{G}_1 , K , S)$-arithmetic subgroup of
$G_1(F).$ Then the set of $K$-isomorphism classes of $K$-forms
$\mathscr{G}_2$ of $\overline{G}_2$ such that $G_2(F)$ contains a
Zariski-dense $(\mathscr{G}_2 , K , S)$-arithmetic subgroup weakly
commensurable to $\Gamma_1$ is finite.

In other words, the set of all Zariski-dense $(K , S)$-arithmetic
subgroups of  $G_2(F)$ which are weakly commensurable to a given
Zariski-dense $(K , S)$-arithmetic subgroup is a union of finitely
many commensurability classes.
\end{thm}

\vskip2mm

A noteworthy fact about weak commensurability is that it has the
following implication for the existence of unipotent elements in
arithmetic subgroups (even though it is formulated entirely in terms
of semi-simple ones). We recall that a semi-simple $K$-group is
called $K$-{\it isotropic} if $\mathrm{rk}_K\: G > 0$; in
characteristic zero, this is equivalent to the existence of
nontrivial unipotent elements in $G(K)$. Moreover, if $K$ is a
number field then $G$ is $K$-isotropic if and only if every
$S$-arithmetic subgroup contains unipotent elements, for any $S$.
\begin{thm}\label{T:WC6}
Let $G_1$  and $G_2$ be two connected absolutely almost simple
algebraic groups defined over a field $F$ of characteristic zero.
For $i = 1,2$, let $\Gamma_i$ be a Zariski-dense $(\mathscr{G}_i ,K
, S)$-arithmetic subgroup of $G_i(F)$. If $\Gamma_1$ and $\Gamma_2$
are weakly commensurable then $\mathrm{rk}_K\mathscr{G}_1
=\mathrm{rk}_K\mathscr{G}_2$; in particular, if $\mathscr{G}_1$ is
$K$-isotropic, then so is $\mathscr{G}_2$.
\end{thm}

We note that in \cite[\S 7]{PR-IHES} we prove a somewhat more
precise result, viz. that if $G_1$ and $G_2$ are of the same type,
then the Tits indices of $\mathscr{G}_1/K$ and $\mathscr{G}_2/K$ are
isomorphic, but we will not get into these technical details here.

\vskip2mm

The following result asserts that a lattice{\footnote{ A discrete
subgroup $\Gamma$ of a locally compact topological group $\cG$ is
said to be a lattice in $\cG$ if $\cG/\Gamma$ carries a {\it finite}
$\cG$-invariant Borel measure.}} which is weakly commensurable with
an $S$-arithmetic group is itself $S$-arithmetic.
\begin{thm}\label{T:WC7}
Let $G_1$ and $G_2$ be two connected absolutely almost simple
algebraic groups defined over a nondiscrete locally compact field
$F$ of characteristic zero, and for $i =1, \,2$, let $\Gamma_i$ be a
Zariski-dense lattice in $G_i(F).$ Assume that $\Gamma_1$ is a $(K ,
S)$-arithmetic subgroup of $G_1(F)$. If $\Gamma_1$ and $\Gamma_2$
are weakly commensurable, then $\Gamma_2$ is a $(K , S)$-arithmetic
subgroup of $G_2(F)$.
\end{thm}

\vskip2mm

\section{Geometric applications}\label{S:WC2}

\vskip2mm

We are now in a position
to give the precise statements of our results on isospectral and
length-commensurable locally symmetric spaces. Throughout this
subsection, for $i = 1, 2$, $G_i$ will denote an absolutely simple
real algebraic group and $\fX_i$ the symmetric space of
$\mathcal{G}_i = G_i(\R)$. Furthermore, given a discrete
torsion-free subgroup $\Gamma_i \subset \mathcal{G}_i$, we let
$\fX_{\Gamma_i} = \fX_i/\Gamma_i$ denote the corresponding locally
symmetric space. The geometric results are basically obtained by
combining Theorem \ref{T:LC-WC} and Corollary \ref{C:LC-WC} with the
results on weakly commensurable subgroups from the previous
section. It should be emphasized that when
$\fX_{\Gamma_1}$ and $\fX_{\Gamma_2}$ are both rank one spaces and
we are not in the exceptional case $(\mathcal{E})$ (which is the
case, for example, for all hyperbolic $n$-manifolds with $n
\geqslant 4$) our results are {\it unconditional}, while in all
other cases they depend on the validity of Schanuel's conjecture.

\vskip1mm

Now, applying Theorems \ref{T:WC1} and \ref{T:WC2} we obtain the
following.
\begin{thm}\label{T:WC8}
Let $G_1$ and $G_2$ be connected absolutely simple real algebraic
groups, and let $\fX_{\Gamma_i}$ be a locally symmetric space of
finite volume, of $\cG_i,$ for $i = 1,\,2.$ If $\fX_{\Gamma_1}$ and
$\fX_{\Gamma_2}$ are length-commensurable, then {\rm{(i)}} either
$G_1$ and $G_2$ are of same Killing-Cartan type, or one of them is
of type $\textsf{B}_{n}$ and the other is of type
$\textsf{C}_{n}$ for some $n \geqslant 3$, {\rm{(ii)}}
$K_{\Gamma_1} = K_{\Gamma_2}.$
\end{thm}

It should be pointed out that assuming Schanuel's conjecture in all
cases, one can prove this theorem (in fact, a much stronger
statement -- see Theorem \ref{T:Fields1}) assuming only that
$\Gamma_1$ and $\Gamma_2$ are finitely generated and Zariski-dense.

\vskip1mm

Next, using Theorems \ref{T:WC4} and \ref{T:WC5} we obtain
\begin{thm}\label{T:WC9}
Let $G_1$ and $G_2$ be connected absolutely simple real algebraic
groups, and let $\cG_i = G_i(\bR),$ for $i = 1,\,2.$ Then the set of
arithmetically defined locally symmetric spaces $\fX_{\Gamma_2}$ of
$\cG_2$, which are length-commensurable to a given arithmetically
defined locally symmetric space $\fX_{\Gamma_1}$ of $\cG_1$, is a
union of finitely many commensurability classes. It in fact consists
of a single commensurability class if $G_1$ and $G_2$ have the same
type different from $\textsf{A}_{n}$, $\textsf{D}_{2n+1}$, with  $n
> 1$, or $\textsf{E}_6.$
\end{thm}

\vskip2mm

Furthermore, Theorems \ref{T:WC6}  and \ref{T:WC7} imply the
following rather surprising result which has so far defied all
attempts of a  purely geometric proof.
\begin{thm}\label{T:WC10}
Let $G_1$ and $G_2$ be connected absolutely simple real algebraic
groups, and let $\fX_{\Gamma_1}$ and $\fX_{\Gamma_2}$ be
length-commensurable locally symmetric spaces of $\cG_1$ and $\cG_2$
respectively, of finite volume. Assume that at least one of the
spaces is arithmetically defined. Then the other space is also
arithmetically defined, and the compactness of one of the spaces
implies the compactness of the other.
\end{thm}

In fact, if one of the spaces is compact and the other is not, the
weak length spectra $L(\fX_{\Gamma_1})$ and $L(\fX_{\Gamma_2})$ are
quite different -- see Theorem \ref{T:Fields5} for a precise
statement (we note that the proof of this result uses Schanuel's
conjecture in all cases).

\vskip2mm

Finally, we will describe some applications to isospectral compact
locally symmetric spaces. So, in the remainder of this section, the
locally symmetric spaces $\fX_{\Gamma_1}$ and $\fX_{\Gamma_2}$ as
above will be assumed to be {\it compact}. Then, as we discussed in
\S \ref{S:Intro}, the fact that $\fX_{\Gamma_1}$ and
$\fX_{\Gamma_2}$ are isospectral implies that $L(\fX_{\Gamma_1}) =
L(\fX_{\Gamma_2})$, so we can use our results on
length-commensurable spaces. Thus,  in particular we obtain the
following.
\begin{thm}\label{T:WC11}
If $\fX_{\Gamma_1}$ and $\fX_{\Gamma_2}$ are isospectral, and
$\Gamma_1$ is arithmetic, then so is $\Gamma_2$.
\end{thm}

Thus, the Laplace spectrum can see if the fundamental group is
arithmetic or not -- to our knowledge, no results of this kind,
particularly for general locally symmetric spaces, were previously
known in spectral theory.

\vskip2mm

The following theorem settles the question ``Can one hear the shape
of a drum?'' for arithmetically defined compact locally symmetric
spaces.
\begin{thm}\label{T:WC12}
Let $\fX_{\Gamma_1}$ and $\fX_{\Gamma_2}$ be compact locally
symmetric spaces associated with absolutely simple real algebraic
groups $G_1$ and $G_2$, and assume that at least one of the spaces
is arithmetically defined. If $\fX_{\Gamma_1}$ and $\fX_{\Gamma_2}$
are isospectral then $G_1 = G_2 := G$. Moreover, unless $G$ is of
type $\textsf{A}_{n}$, $\textsf{D}_{2n+1}$ $(n > 1)$, or
$\textsf{E}_6$, the spaces $\fX_{\Gamma_1}$ and $\fX_{\Gamma_2}$ are
commensurable.
\end{thm}

It should be noted that our methods based on length-commensurability
or weak commensurability leave room for the following ambiguity in
the proof of Theorem \ref{T:WC12}: either $G_1 = G_2$ or $G_1$ and
$G_2$ are $\R$-split forms  of types $\textsf{B}_{n}$ and
$\textsf{C}_{n}$ for some $n \geqslant 3$ - and this ambiguity is
unavoidable, cf.\, the end of \S \ref{S:BC}. The fact that in the
latter case the locally symmetric spaces cannot be isospectral was
shown by Sai-Kee Yeung \cite{SKY} by comparing the traces of the
heat operator (without using Schanuel's conjecture), which leads to
the statement of the theorem given above.

\vskip3mm

\section{Absolutely almost simple algebraic groups having the same
maximal tori}\label{S:Tori}

\vskip2mm

The analysis of weak commensurability is related to another natural
problem in the theory of algebraic groups of characterizing
absolutely almost simple $K$-groups having the same
isomorphism/isogeny classes of maximal $K$-tori -- the exact nature
of this connection will be clarified in Theorem \ref{T:Gen4} and the
subsequent discussion.  Some aspects of this problem over local and
global fields were considered in \cite{Garge} and \cite{Kar}.
Another direction of research, which has already generated a number
of results (cf.\,\cite{Bayer}, \cite{Gar}, \cite{Lee},
\cite{PR-invol}) is the investigation of local-global principles for
embedding tori into absolutely almost simple algebraic groups as
maximal tori (in particular, for embedding of  commutative \'etale
algebras with involutive automorphisms into simple algebras with
involution); some number-theoretic applications of these results can
be found, for example, in \cite{Fiori}. A detailed discussion of
these issues would be an independent undertaking, so we will limit
ourselves here to the following theorem (cf.\,\cite[Theorem
7.5]{PR-IHES} and \cite[Proposition 1.3]{GarR}).
\begin{thm}\label{T:WC13}
{\rm (1)} Let $G_1$ and $G_2$ be connected absolutely almost simple
algebraic groups defined over a number field $K$, and let $L_i$ be
the smallest Galois extension of $K$ over which $G_i$ becomes an
inner form of a split group. If $G_1$ and $G_2$ have the same
$K$-isogeny classes of maximal $K$-tori then either $G_1$ and $G_2$
are of the same Killing-Cartan type, or one of them is of type
$\textsf{B}_{n}$ and the other is of type $\textsf{C}_{n}$,
and moreover, $L_1 = L_2$.

\vskip2mm

\noindent {\rm (2)} \parbox[t]{13.8cm}{Fix an absolutely almost
simple $K$-group $G$. Then the set of isomorphism classes of all
absolutely almost simple $K$-groups $G'$ having the same $K$-isogeny
classes of maximal $K$-tori is finite.}

\vskip2mm

\noindent {\rm (3)} \parbox[t]{13.8cm}{Fix an absolutely almost
simple simply connected $K$-group $G$ whose Killing-Cartan type is
different from $\textsf{A}_{n}$, $\textsf{D}_{2n+1}$ $(n >
1)$ or $\textsf{E}_6$. Then any $K$-form $G'$ of $G$ (in other
words, any absolutely almost simple simply connected $K$-group $G'$
of the \emph{same} type as $G$) that has the same $K$-isogeny
classes of maximal $K$-tori as $G$, is isomorphic to $G$.}
\end{thm}

The construction described in \cite[\S 9]{PR-IHES} shows that the
types excluded in (3) are honest exceptions, i.e., for each of those
types one can construct non-isomorphic absolutely almost simple
simply connected $K$-groups $G_1$ and $G_2$ of this type over a
number field $K$ that have the same isomorphism classes of maximal
$K$-tori. On the other hand, the analysis of the situation where
$G_1$ and $G_2$ are of types $\textsf{B}_{n}$ and $\textsf{C}_{n}$,
respectively, over a number field $K$ and have the same
isomorphism/isogeny classes of maximal $K$-tori is given in Theorem
\ref{T:BC2} below (cf. \cite[Theorem 1.4 and 1.5]{GarR})

Of course, the question about determining absolutely almost simple algebraic
$K$-groups by their maximal $K$-tori makes sense over general
fields. It is particularly interesting for division algebra where it
can be reformulated as the following question which is somewhat
reminiscent of Amitsur's famous theorem on generic splitting fields
(cf.\,\cite{Ami}, \cite{GiSz}): {\it What can one say about two
finite-dimensional central division algebras $D_1$ and $D_2$ over
the same field $K$ given the fact that they have the same
isomorphism classes of maximal subfields?} For recent results on
this problem see \cite{CR2}, \cite{GarS}, \cite{KraMcK}, \cite{R2}.

\section{Weakly commensurable subgroups in groups of types
$\textsf{B}$ and $\textsf{C}$}\label{S:BC}

Let $G_1$ and $G_2$ be absolutely almost simple algebraic groups
over a field $K$ of characteristic zero. According to Theorem
\ref{T:WC1}, finitely generated weakly commensurable Zariski-dense
subgroups $\Gamma_1 \subset G_1(K)$ and $\Gamma_2 \subset G_2(K)$
can exist only if $G_1$ and $G_2$ are of the same Killing-Cartan
type or one of them is of type $\textsf{B}_{n}$ and the other is of
type $\textsf{C}_{n}$ for some $n \geqslant 3$. Moreover, the
results we described in \S \ref{S:WC-Res} provide virtually complete
answers to the key questions about weakly commensurable
$S$-arithmetic subgroups in the case where $G_1$ and $G_2$ are {\it
of the same type}. In this section, we will discuss recent results
\cite{GarR} that determine weakly commensurable arithmetic subgroups
when $G_1$ is {\it of type} $\textsf{B}_{n}$ and $G_2$ is {\it of
type} $\textsf{C}_{n}$ $(n \geqslant 3)$.

\vskip2mm

First of all, it should be pointed out that $S$-arithmetic subgroups
in groups of types $\textsf{B}_{n}$ and $\textsf{C}_{n}$ can indeed
be weakly commensurable. The underlying reason is that if $G_1$ is
a split adjoint  group of type $\textsf{B}_{n}$ and $G_2$ is a
split simply connected group of type $\textsf{C}_{n}$ $(n \geqslant
2)$ over any field $K$ of characteristic $\neq 2$,  then $G_1$ and
$G_2$ have the {\it same isomorphism classes of maximal} $K$-{\it
tori}. For the reader's convenience we briefly recall the
Galois-cohomological proof of this fact given in \cite[Example
6.7]{PR-IHES}.

It is well-known that for any semi-simple $K$-group $G$ there is a
natural bijection between the set of $G(K)$-conjugacy classes of
maximal $K$-tori of $G$ and the set
$$\sC_K := \mathrm{Ker}\left(H^1(K , N) \to H^1(K , G)\right),$$
where $T$ is a maximal $K$-torus of $G$ and  $N$ is the normalizer
of $T$ in $G$ (cf.\,\cite[Lemma 9.1]{PR-IHES}). Let $W = N/T$ be the
corresponding Weyl group and introduce the following natural maps in
Galois cohomology:
$$
\theta_K \colon H^1(K , N) \to H^1(K , W) \ \ \ \text{and} \ \ \
\nu_K \colon H^1(K , W) \to H^1(K , \mathrm{Aut}\: T).
$$
To apply  these considerations to the groups $G_1$ and $G_2$, we
will denote by $T_i$ a fixed maximal $K$-split torus of $G_i$ and
let $N_i$, $W_i$, $\sC^{(i)}_K$, $\theta^{(i)}_K$ and  $\nu_K^{(i)}$ be
the corresponding objects attached to $G_i$. It follows from an explicit
description of the root systems of types $\textsf{B}_n$ and
$\textsf{C}_n$ that there exist $K$-isomorphisms $\varphi
\colon T_1 \to T_2$ and $\psi \colon W_1 \to W_2$ such that for the
natural action of $W_i$ on $T_i$ we have
$$
\varphi(w \cdot t) = \psi(w) \cdot \varphi(t) \ \ \ \text{for all} \
\ t \in T_1, \ w \in W_1.
$$
Since $G_i$ is $K$-split, we have $\theta^{(i)}_K(\sC^{(i)}_K) =
H^1(K , W_i)$ (cf.\,\cite{Gille}, \cite{Kot}, \cite{Rag}). So, $\psi$
induces a natural bijection between $\theta^{(1)}_K(\sC^{(1)}_K)$
and $\theta^{(2)}_K(\sC^{(2)}_K)$. Finally, we observe that if $S_i$
is a maximal $K$-torus of $G_i$ in the $G_i(K)$-conjugacy class
corresponding to $c_i \in \sC^{(i)}_K$, then the $K$-isomorphism
class of $S_i$ is determined by
$\nu^{(i)}_K(\theta^{(i)}_K(c_i))$, and if
$\psi(\theta^{(1)}_K(c_1)) = \theta^{(2)}_K(c_2)$ then $S_1$ and
$S_2$ are $K$-isomorphic. It follows that $G_1$ and $G_2$ have the
same classes of maximal $K$-tori, as required.

\vskip2mm

Subsequently, in \cite{GarR} a more explicit explanation of this
fact was given. More precisely, let $A$ be a central simple algebra
over $K$ with a $K$-linear involution $\tau$ (involution of the
first kind). We recall that $\tau$ is called {\it orthogonal} if
$\dim_K A^{\tau} = n(n + 1)/2$ and {\it symplectic} if $\dim_K
A^{\tau} = n(n - 1)/2$. Furthermore, if $\tau$ is orthogonal and $n
= 2m + 1$ $(m \geqslant 2)$ then $A = M_n(K)$ and the corresponding
algebraic group $G = \mathrm{SU}(A , \tau)$ coincides with the
orthogonal group $\mathrm{SO}_n(q)$ of a nondegenerate
$n$-dimensional quadratic form $q = q_{\tau}$ over $K$, hence  is a
simple adjoint algebraic $K$-group of type $\textsf{B}_{m}$ (note
that the $K$-rank of $G$ equals the Witt index of $q$). If $\tau$ is
symplectic then necessarily $n = 2m$ and $G = \mathrm{SU}(A , \tau)$
is an almost simple simply connected $K$-group of type
$\textsf{C}_{m}$; moreover $G$ is $K$-split  if and only if $A =
M_n(K)$, in which case $G$ is of course isomorphic to
$\mathrm{Sp}_{2m}$. Next, in all cases, any maximal $K$-torus $T$ of
$G$ has the form $T = \mathrm{SU}(E , \sigma)$ where $E$ is a
$\tau$-invariant $n$-dimensional commutative \'etale $K$-subalgebra
of $A$ such that for $\sigma = \tau \vert E$ we have
\begin{equation}\label{E:dim}
\dim E^{\sigma} = \left[  \frac{n+1}{2} \right].
\end{equation}
So, the question whether $G = \mathrm{SU}(A , \tau)$, with $A$ and
$\tau$ as above, has a maximal $K$-torus of a specific type can be
reformulated as follows: Let $(E , \sigma)$ be an $n$-dimensional
commutative \'etale $K$-algebra with an involutive $K$-automorphism $\sigma$
satisfying (\ref{E:dim}). When does there exist an embedding $(E ,
\sigma) \hookrightarrow (A , \tau)$ as algebras with involution?
While in the general case this question is nontrivial (cf.
\cite{PR-invol}), the answer in the case where the group $G$ splits
over $K$ is quite straightforward.
\begin{prop}\label{P:Stein}
{\rm (\cite[2.3 and 2.5]{GarR})} Let $A = M_n(K)$ with a $K$-linear
involution $\tau$, and let $(E , \sigma)$ be an $n$-dimensional
commutative \'etale
$K$-algebra with involution satisfying (\ref{E:dim}). In each of the
following situations

\vskip2mm

\noindent {\rm (1)} $\tau$ is symplectic,

\vskip1mm

\noindent {\rm (2)} \parbox[t]{13.5cm}{$n = 2m + 1$ and $\tau$ is
orthogonal such that the corresponding quadratic form $q_{\tau}$ has
Witt index $m$,}

\vskip1.5mm

\noindent there exists a $K$-embedding $(E , \sigma) \hookrightarrow
(A , \tau)$.
\end{prop}

(This proposition should be viewed as an analogue for algebras with
involution of the following result of Steinberg \cite{Stein}: Let
$G_0$ be a {\it quasi-split} simply connected almost simple
algebraic group over a field $K$. Then given an inner form $G$ of
$G_0$, {\it any maximal} $K$-{\it torus} $T$ {\it of} $G$ {\it
admits a} $K$-{\it embedding into} $G_0$. While the proof of this
result is rather technical, the proof of Proposition \ref{P:Stein},
as well as of the corresponding assertion for the algebras with
involution involved in the description of algebraic groups of type
$\textsf{A}$ and $\textsf{D}$ is completely elementary -- see
\cite[\S 2]{GarR}.)

Now, fix $n \geqslant 2$ and let $A_1 = M_{n_1}(K)$, where $n_1 = 2n
+ 1$, with an orthogonal involution $\tau_1$ such that the Witt
index of the corresponding quadratic form $q_{\tau_1}$ is $n$, and
let $A_2 = M_{n_2}(K)$, where $n_2 = 2n$, with a symplectic
involution $\tau_2$. According to Proposition \ref{P:Stein}, for $i
= 1, 2$, the maximal $K$-tori of $G_i = \mathrm{SU}(A_i , \tau_i)$
are of the form $T_i = \mathrm{SU}(E_i , \sigma_i)$ for a commutative \'etale $K$-algebra $E_i$
of dimension $n_i$ with an involution
$\sigma_i$ satisfying (\ref{E:dim}). On the other hand, the
correspondence
$$
(E_2 , \sigma_2) \mapsto (E_1 , \sigma_1) := (E_2 \times K \, , \,
\sigma_2 \times \mathrm{id}_K)
$$
defines a natural bijection between the isomorphism classes of commutative \'etale
$K$-algebras with involution satisfying (\ref{E:dim}), of dimension
$n_2$ and $n_1$, respectively. Since $\mathrm{SU}(E_1 , \sigma_1) =
\mathrm{SU}(E_2 , \sigma_2)$ in these notations, we again obtain
that $G_1$ and $G_2$ have the same isomorphism classes of maximal
$K$-tori (cf.\,\cite{GarR}, Remark 2.6).

\vskip2mm

Now, let $K$ be a number field and $S$ be any finite set of places
of $K$ containing the set $V_{\infty}^K$ of archimedean places.
Furthermore, let $G_1$ be a split adjoint $K$-group of type
$\textsf{B}_{n}$, and $G_2$ be a split simply connected $K$-group of
type $\textsf{C}_{n}$ $(n \geqslant 2)$. Then the fact, discussed
above, that $G_1$ and $G_2$ have the same isomorphism classes of
maximal $K$-tori immediately implies that the $S$-arithmetic
subgroups in $G_1$ and $G_2$ are weakly commensurable
(cf.\,\cite[Examples 6.5 and 6.7]{PR-IHES}).

\vskip2mm

A complete determination of weakly commensurable $S$-arithmetic
subgroups in  algebraic groups $G_1$ and $G_2$ of types
$\textsf{B}_n$ and $\textsf{C}_n$ $(n \geqslant 3)$ respectively was
recently obtained  by Skip Garibaldi and the second-named author
\cite{GarR}. To formulate the result we need the following {
definition. Let $\sG_1$ and $\sG_2$ be absolutely almost simple
algebraic groups of types $\textsf{B}_{n}$ and $\textsf{C}_{n}$ with
$n \geqslant 2$, respectively, over a number field $K$. We say that
$\sG_1$ and $\sG_2$ are {\it twins} (over $K$) if for each place $v$
of $K$, either both groups are split or both are anisotropic over
the completion $K_v$. (We note that since groups of these types
cannot be anisotropic over $K_v$ when $v$ is nonarchimedean, our
condition effectively says that $\sG_1$ and $\sG_2$ must be
$K_v$-split for {\it all} nonarchimedean $v$.)

\begin{thm}\label{T:BC1}
{\rm (\cite[Theorem 1.2]{GarR})} Let $G_1$ and $G_2$ be absolutely
almost simple algebraic groups over a field $F$ of characteristic
zero having Killing-Cartan types $\textsf{B}_{n}$ and
$\textsf{C}_{n}$ $(n \geqslant 3)$ respectively, and let
$\Gamma_i$ be a Zariski-dense $(\sG_i, K, S)$-arithmetic subgroup of
$G_i(F)$ for $i = 1, 2$. Then $\Gamma_1$ and $\Gamma_2$ are weakly
commensurable if and only if the groups $\sG_1$ and $\sG_2$ are
twins.
\end{thm}

(We recall that according to Theorem \ref{T:WC3}, if Zariski-dense
$(\sG_1, K_1, S_1)$- and $(\sG_2, K_2, S_2)$-arithmetic  subgroups
are weakly commensurable then necessarily $K_1 = K_2$ and $S_1 =
S_2$, so Theorem \ref{T:BC1} in fact treats the most general
situation.)

\vskip2mm

The necessity is proved using generic tori (cf.\,\S \ref{S:Gen}) in
conjunction with the analysis of maximal tori in real groups of
types $\textsf{B}_{n}$ and $\textsf{C}_{n}$ (which can also be found
in \cite{DjTh}). The proof of sufficiency is obtained using the
above description of maximal $K$-tori in terms of commutative \'etale $K$-subalgebras
with involution and the local-global results for the existence of an
embedding of commutative \'etale algebras with involution into simple algebras
with involution established in \cite{PR-invol}; an alternative
argument along the lines outlined in the beginning of this section
can be given using Galois cohomology of algebraic groups (cf.\,\cite[\S 9]{GarR}).

\vskip2mm

As we already mentioned in \S \ref{S:Tori}, the analysis of weak
commensurability involved in the proof of Theorem \ref{T:BC1} leads
to, and at the same time depends on, the following result describing
when groups of types $\textsf{B}_{n}$ and $\textsf{C}_{n}$ have the
same isogeny/isomorphism classes of maximal $K$-tori.

\begin{thm}\label{T:BC2}
{\rm (\cite[Theorem 1.4]{GarR})} Let $G_1$ and $G_2$ be absolutely
almost simple algebraic groups over a number field $K$ of types
$\textsf{B}_{n}$ and $\textsf{C}_{n}$ respectively for some
$n \geqslant 3$.

\vskip2mm

\noindent {\rm (1)} \parbox[t]{13.5cm}{The groups $G_1$ and $G_2$
have the same \emph{isogeny} classes of maximal $K$-tori if and only
if they are twins.}

\vskip1mm

\noindent {\rm (2)} \parbox[t]{13.5cm}{The groups $G_1$ and $G_2$
have the same \emph{isomorphism} classes of maximal $K$-tori
 if and only if they are
twins, $G_1$ is adjoint, and $G_2$ is simply connected.}
\end{thm}

Theorem \ref{T:BC1} has the following interesting geometric
applications \cite{PR-Fields}. Again, let $G_1$  and $G_2$ be simple
real algebraic groups of types $\textsf{B}_{n}$ and $\textsf{C}_{n}$
respectively. For $i = 1, 2$, let $\Gamma_i$ be a discrete
torsion-free $(\sG_i , K)$-arithmetic subgroup of $\cG_i = G_i(\R)$,
and let $\fX_{\Gamma_i}$ be the corresponding locally symmetric
space of $\cG_i$. Then if $\sG_1$ and $\sG_2$ are not twins, the
locally symmetric spaces $\fX_{\Gamma_1}$ and $\fX_{\Gamma_2}$ are
{\it not} length-commensurable. As one application of this result,
we would like to point out the following assertion: {\it Let $M_1$
be an arithmetic quotient of the real hyperbolic space
$\mathbb{H}^p$ $(p \geqslant 5)$ and $M_2$ be an arithmetic quotient
of the quaternionic hyperbolic space $\mathbb{H}^q_{\mathbf{H}}$ $(q
\geqslant 2)$. Then $M_1$ and $M_2$ are not length-commensurable.}
The results of \cite{GarR} are used to handle the case $p = 2n$ and
$q = n - 1$ for some $n \geqslant 3$; for other values of $p$ and
$q$, the claim follows from Theorem \ref{T:WC8}.

On the other hand, suppose $G_1 = \mathrm{SO}(n + 1 , n)$ and $G_2 =
\mathrm{Sp}_{2n}$ over $\R$ (i.e., $G_1$ and $G_2$ are the
$\R$-split forms of types $\textsf{B}_{n}$ and $\textsf{C}_{n}$,
respectively) for some $n \geqslant 3$. Furthermore, let $\Gamma_i$
be a discrete torsion-free $(\sG_i , K)$-arithmetic subgroup of
$\cG_i$ for $i = 1, 2$, and let $M_i = \fX_{\Gamma_i}$. If $\sG_1$
and $\sG_2$ are twins then
$$
\Q \cdot L(M_2) = \lambda \cdot \Q \cdot L(M_1) \ \ \text{where} \ \
\lambda = \sqrt{\frac{2n + 2}{2n - 1}}.
$$
Thus, despite the fact that they are associated with groups of
different types, the locally symmetric spaces $M_1$ and $M_2$ can be
made length-commensurable by scaling the metric on one of them;
this, however, will {\it never} make them isospectral \cite{SKY}.
What is interesting is that so far this is the {\it only}  situation
in our analysis commensurability of isospectral and
length-commensurable locally symmetric spaces where isospectrality
manifests itself as an essentially stronger condition.

\vskip4mm

\section{On the fields generated by the lengths of closed
geodesics}\label{S:Fields}

\vskip4mm

In \S \ref{S:WC2} (and also  at the end of \S \ref{S:BC}) we
discussed the consequences of length-commensu- rability of two
locally symmetric spaces $M_1$ and $M_2$; our focus in this section
will be on the consequences of {\it non-length-commensurability} of
$M_1$ and $M_2$. More precisely, we will explore how different in
this case  the sets $L(M_1)$ and $L(M_2)$ (or $\Q \cdot L(M_1)$ and
$\Q \cdot L(M_2)$) are and whether they can  in fact be related in
{\it any} reasonable way? Of course, one can ask a number of
specific questions that fit this general perspective: for example,
can $L(M_1)$ and $L(M_2)$ differ only in a finite number of
elements, in other words, can the symmetric difference $L(M_1)
\bigtriangleup L(M_2)$ be finite? Or can it happen that
$\overline{\Q} \cdot L(M_1) = \overline{\Q} \cdot L(M_2)$, where
$\overline{\Q}$ is the field of all algebraic numbers; in other
words, can the use of the field $\overline{\Q}$  in place of $\Q$ in
the definition of length-commensurability essentially change this
relation? One relation between $L(M_1)$ and $L(M_2)$ that would make
a lot of sense geometrically is that of similarity, requiring that
there be a real number $\alpha > 0$ such that
$$
L(M_2) = \alpha \cdot L(M_1) \ \ \ (\text{or} \ \ \Q \cdot L(M_2) =
\alpha \cdot \Q \cdot L(M_1) \ ),
$$
which means that $M_1$ and $M_2$ can be made iso-length-spectral
(resp., length-commensu- rable) by scaling the metric on one of
them. From the algebraic standpoint, one can generalize this
relation by considering arbitrary polynomial relations between
$L(M_1)$ and $L(M_2)$ instead of just linear relations although this
perhaps does not have a clear geometric interpretation. To formalize
this general idea, we need to introduce some additional notations
and definitions.

\vskip2mm

For a Riemannian manifold $M$, we let $\mathscr{F}(M)$ denote the
subfield of $\R$ generated by the set $L(M)$. Given two Riemannian
manifolds $M_1$ and $M_2$, we set $\mathscr{F}_i = \mathscr{F}(M_i)$
for $i \in \{1 , 2\}$ and consider the following condition

\vskip2mm

\noindent $(T_i)$ \  {\it the compositum
$\mathscr{F}_1\mathscr{F}_2$ has infinite transcendence degree over
the field $\mathscr{F}_{3-i}$.}

\vskip2mm

\noindent Informally, this condition means that $L(M_i)$ contains
``many" elements which are algebraically independent from all the
elements of $L(M_{3-i})$, implying the nonexistence of any
nontrivial polynomial dependence between $L(M_1)$ and $L(M_2)$. In
particular, $(T_i)$ implies the following condition

\vskip2mm

\noindent \ $(N_i)$ \ \,$L(M_i) \not\subset A \cdot \Q \cdot L(M_{3-i})$
{\it for any finite set $A$ of real numbers.}

\vskip2mm

In \cite{PR-Fields}, we have proved a series of results asserting
that if $M_i = \fX_{\Gamma_i}$ for $i = 1, 2$, are the quotients of
symmetric spaces $\fX_i$ associated with absolutely simple real
algebraic groups $G_i$ by Zariski-dense discrete torsion-free
subgroups $\Gamma_i \subset G_i(\R)$, then in many situations the
fact that $M_1$ and $M_2$ are not length-commensurable implies that
conditions $(T_i)$ and $(N_i)$ hold for at least one $i \in \{1 ,
2\}$. To give precise formulations, in addition to the standard
notations used earlier, we let $w_i$ denote the order of the
(absolute) Weyl group of $G_i$. We also need to emphasize that {\it
all geometric results in} \cite{PR-Fields} {\it assume the validity
of Shanuel's conjecture.} This assumption, however, enables one to
establish results that are somewhat stronger than the corresponding
results in \S \ref{S:WC2} and do not require that
$\Gamma_1$ and $\Gamma_2$ have property $(A)$ (2.2). We begin with
the following result which strengthens Theorem \ref{T:WC8}.

\begin{thm}\label{T:Fields1}
Assume that the Zariski-dense subgroups $\Gamma_1$ and $\Gamma_2$
are finitely generated (which is automatically the case if these
subgroups are lattices).

\vskip1mm

\noindent {\rm (1)} If $w_1 > w_2$ then $(T_1)$ holds;

\vskip.3mm

\noindent {\rm (2)} If $w_1 = w_2$ but $K_{\Gamma_1} \not\subset
K_{\Gamma_2}$ then again $(T_1)$ holds.

\vskip1mm

\noindent Thus, unless $w_1 = w_2$ and $K_{\Gamma_1} =
K_{\Gamma_2}$, the condition $(T_i)$ holds for at least one $i \in
\{1 , 2\}$; in particular, $M_1$ and $M_2$ are not
length-commensurable.
\end{thm}

As follows from Theorem \ref{T:Fields1}, we only need to consider
the case where $w_1 = w_2$, which we will assume -- recall that this
entails that either $G_1$ and $G_2$ are of the same Killing-Cartan
type, or one of them is of type $\textsf{B}_{n}$  and the other is
of type $\textsf{C}_{n}$  $(n \geqslant 3)$. Then it is convenient
to divide our results for {\it arithmetic} subgroups $\Gamma_1$ and
$\Gamma_2$ into three theorems: the first one will treat the case
where $G_1$ and $G_2$ are of the {\it same type} which is different
from $\textsf{A}_{n}$, $\textsf{D}_{2n + 1}$ $(n
> 1)$ and $\textsf{E}_6$, the second one -- the case where both $G_1$ and
$G_2$ are one of the  types $\textsf{A}_{n}$, $\textsf{D}_{2n + 1}$ $(n
> 1)$ and $\textsf{E}_6$, and the third one
-- the case where $G_1$ is of type $\textsf{B}_{n}$ and $G_2$ is
of type $\textsf{C}_{n}$ for some $n \geqslant 3$.

\begin{thm}\label{T:Fields2}
With notations as above, assume that $G_1$ and $G_2$ are of the same
Killing-Cartan type which is different from $\textsf{A}_{n}$,
$\textsf{D}_{2n + 1}$ $(n > 1)$ and $\textsf{E}_6$ and that the
subgroups $\Gamma_1$ and $\Gamma_2$ are arithmetic. Then either $M_1
= \fX_{\Gamma_1}$ and $M_2 = \fX_{\Gamma_2}$ are commensurable
(hence also length-commensurable), or conditions $(T_i)$ and $(N_i)$
hold for at least one $i \in \{1 , 2\}$.
\end{thm}

(This theorem strengthens part of Theorem \ref{T:WC9}. We also note
that $(T_i)$ and $(N_i)$ may not hold for {\it both} $i = 1$ and
$2$; in fact, it is possible that one of $L(M_1)$ and $L(M_2)$ is contained in the other.)

\vskip2mm

\begin{thm}\label{T:Fields3}
Again, keep the above notations and assume that the common
Killing-Cartan type of $G_1$ and $G_2$ is one of the following:
$\textsf{A}_{n}$, $\textsf{D}_{2n + 1}$ $(n > 1)$ or
$\textsf{E}_6$ and that the subgroups $\Gamma_1$ and $\Gamma_2$ are
arithmetic. Assume in addition that $K_{\Gamma_i} \neq \Q$ for at
least one $i \in \{1 , 2\}$. Then either $M_1$ and $M_2$ are
length-commensurable (although not necessarily commensurable), or
conditions $(T_i)$ and $(N_i)$ hold for at least one $i \in \{1 ,
2\}$.
\end{thm}

To illustrate possible applications of these theorems, we will now
give explicit statements for real hyperbolic manifolds -- cf.
Theorem \ref{T:Hyper}; similar results are available for complex and
quaternionic hyperbolic spaces.

\begin{cor}\label{C:Fields1}
Let $M_i$ $(i = 1, 2)$ be the quotients of the real hyperbolic space
$\mathbb{H}^{d_i}$ with $d_i \neq 3$ by a torsion-free Zariski-dense
discrete subgroup $\Gamma_i$ of $G_i(\R)$, where $G_i = \mathrm{PSO}(d_i , 1)$.

\vskip1mm

\ $(i)$ If $d_1 > d_2$ then conditions $(T_1)$ and $(N_1)$ hold.

\vskip.5mm

$(ii)$ If $d_1 = d_2$ but $K_{\Gamma_1} \not\subset K_{\Gamma_2}$
then again conditions $(T_1)$ and $(N_1)$ hold.

\vskip1mm

\noindent Thus, unless $d_1 = d_2$ and $K_{\Gamma_1} =
K_{\Gamma_2}$, conditions $(T_i)$ and $(N_i)$ hold for at least one
$i \in \{1 , 2\}$.

\vskip2mm

\noindent Assume now that $d_1 = d_2 =: d$ and the subgroups
$\Gamma_1$ and $\Gamma_2$ are arithmetic.

\vskip1mm

$(iii)$ \parbox[t]{13cm}{If $d$ is either even or is congruent to
$3(\mathrm{mod}\: 4)$, then either $M_1$ and $M_2$ are
commensurable, hence length-commensurable, or $(T_i)$ and $(N_i)$
hold for at least one $i \in \{1 , 2\}$.}

\vskip.5mm

$(iv)$ \parbox[t]{13cm}{If $d \equiv 1(\mathrm{mod}\: 4)$ and in
addition $K_{\Gamma_i} \neq \Q$ for at least one $i \in \{1 , 2\}$
then either $M_1$ and $M_2$ are length-commensurable (although not
necessarily commensurable), or conditions $(T_i)$ and $(N_i)$ hold
for at least one $i \in \{1 , 2\}$.}
\end{cor}

Now, we consider the case where one of the groups is of type
$\textsf{B}_{n}$ and the other of type $\textsf{C}_{n}$ $(n
\geqslant 3)$. The theorem below strengthens the results of \S
\ref{S:BC}.

\begin{thm}\label{T:Fields4}
Notations as above, assume that $G_1$ is of type $\textsf{B}_{n}$
and $G_2$ is of type $\textsf{C}_{n}$ for some $n \geqslant 3$
and the subgroups $\Gamma_1$ and $\Gamma_2$ are arithmetic. Then
either $(T_i)$ and $(N_i)$ hold for at least one $i \in \{1 , 2\}$,
or
$$
\Q \cdot L(M_2) = \lambda \cdot \Q \cdot L(M_1) \ \ \text{where} \ \
\lambda = \sqrt{\frac{2n + 2}{2n - 1}}.
$$
\end{thm}

\vskip2mm

The following interesting result holds for all types (cf.\,Theorem
\ref{T:WC10}).

\begin{thm}\label{T:Fields5}
For $i = 1, 2$, let $M_i = \fX_{\Gamma_i}$ be an arithmetically
defined locally symmetric space, and assume that $w_1 = w_2$. If
$M_2$ is compact and $M_1$ is not, then conditions $(T_1)$ and
$(N_1)$ hold.
\end{thm}

Finally, we have the following result which shows that the notion
of ``length-similarity" for arithmetically defined locally symmetric
spaces is redundant if $G_1$ and $G_2$ are of the same type (cf.,
however, Theorem \ref{T:Fields4} regarding the case where $G_1$ and
$G_2$ are of types $\textsf{B}_{n}$ and $\textsf{C}_{n}$).

\begin{cor}\label{C:Fields2}
Let $M_i = \fX_{\Gamma_i}$, for $i = 1, 2$, be arithmetically
defined locally symmetric spaces. Assume that there exists $\lambda
\in \R_{> 0}$ such that
$$
\Q \cdot L(M_1) = \lambda \cdot \Q \cdot L(M_2).
$$
Then

\vskip1mm

\ $(i)$ \parbox[t]{13cm}{if $G_1$ and $G_2$ are of the same type
which is different from $\textsf{A}_{n}$, $\textsf{D}_{2n+1}$
$(n > 1)$ and $\textsf{E}_6$, then $M_1$ and $M_2$ are
commensurable, hence length-commensurable;}

\vskip1mm

$(ii)$ \parbox[t]{13cm}{if $G_1$ and $G_2$ are of the same type
which is one of the following: $\textsf{A}_{n}$,
$\textsf{D}_{2n+1}$ $(n > 1)$ or $\textsf{E}_6$, then,
provided that $K_{\Gamma_i} \neq \Q$ for at least one $i \in \{1 ,
2\}$, the spaces $M_1$ and $M_2$ are length-commensurable (although
not necessarily commensurable).}
\end{cor}

The proofs of the results in this section use a generalization of
the notion of weak commensurability which we termed {\it weak
containment}. To give a precise definition, we temporarily return to
the general set-up where $G_1$ and $G_2$ are semi-simple algebraic
groups defined over a field $F$ of characteristic zero, and
$\Gamma_i$ is a Zariski-dense subgroup of $G_i(F)$ for $i = 1,
2$.

\vskip2mm

\noindent {\bf 8.8. Definition.} (a) Semi-simple elements
$\gamma^{(1)}_1, \ldots , \gamma^{(1)}_{m_1} \in \Gamma_1$ are {\it
weakly contained} in $\Gamma_2$ if there are semi-simple elements
$\gamma^{(2)}_1, \ldots , \gamma^{(2)}_{m_2}$ such that
\begin{equation}\label{E:WCont}
\chi^{(1)}_1(\gamma^{(1)}_1) \cdots
\chi^{(1)}_{m_1}(\gamma^{(1)}_{m_1}) = \chi^{(2)}_1(\gamma^{(2)}_1)
\cdots \chi^{(2)}_{m_2}(\gamma^{(2)}_{m_2}) \neq 1
\end{equation}
for some maximal $F$-tori $T^{(j)}_k$ of $G_j$ whose group of $F$-rational points contains
$\gamma^{(j)}_k$ and some characters $\chi^{(j)}_k$ of $T^{(j)}_k$
for $j \in \{ 1 , 2 \}$ and $k \leqslant m_j$. \footnote{We note
that (\ref{E:WCont}) means that the subgroups of
$\overline{F}^{\times}$ generated by the eigenvalues of
$\gamma^{(1)}_1, \ldots , \gamma^{(1)}_{m_1}$ and by those of
$\gamma^{(2)}_1, \ldots , \gamma^{(2)}_{m_2}$ for some
(equivalently, any) matrix realizations of $G_1 \subset
\mathrm{GL}_{N_1}$ and $G_2 \subset \mathrm{GL}_{N_2}$, intersect
nontrivially.}

\vskip1mm

(b) Semi-simple elements $\gamma^{(1)}_1, \ldots ,
\gamma^{(1)}_{m_1} \in \Gamma_1$ are {\it multiplicatively
independent} if for some (equivalently, any) choice of maximal
$F$-tori $T^{(1)}_i$ of $G_1$ such that $\gamma^{(1)}_i \in T_i(F)$
for $i \leqslant m_1$, a relation of the form
$$
\chi^{(1)}_1(\gamma^{(1)}_1) \cdots
\chi^{(1)}_{m_1}(\gamma^{(1)}_{m_1}) = 1,
$$
where $\chi_i \in X(T_i)$ implies that
$$
\chi^{(1)}_1(\gamma^{(1)}_1) = \cdots =
\chi^{(1)}_{m_1}(\gamma_{m_1}) = 1.
$$

\vskip1mm

(c) We say that $\Gamma_1$ and $\Gamma_2$ as above satisfy {\it
property $(C_i)$}, where $i = 1$ or $2$, if for any $m \geqslant 1$
there exist semi-simple elements $\gamma^{(i)}_1, \ldots ,
\gamma^{(i)}_m \in \Gamma_i$ of infinite order that are
multiplicatively independent and are {\it not} weakly contained in
$\Gamma_{3-i}$.

\vskip2mm

Using Schanuel's conjecture, we prove (cf.\,\cite[Corollary
7.3]{PR-Fields}) that {\it if $\fX_{\Gamma_1}$ and $\fX_{\Gamma_2}$
are locally symmetric spaces as above with finitely generated
Zariski-dense fundamental groups $\Gamma_1$ and $\Gamma_2$, then the
fact that these groups satisfy property $(C_i)$ for some $i \in \{1
, 2\}$ implies that the locally symmetric spaces satisfy conditions
$(T_i)$ and $(N_i)$ for the same $i$.} So, the way we prove Theorems
\ref{T:Fields1}-\ref{T:Fields3} and \ref{T:Fields4}-\ref{T:Fields5}
is by showing that condition $(C_i)$ holds in the respective
situations. For example, Theorem \ref{T:Fields1} is a consequence of
the following algebraic result.

\addtocounter{thm}{1}

\begin{thm}\label{T:Fields6}
Assume that $\Gamma_1$ and $\Gamma_2$ are finitely generated (and
Zariski-dense).

\vskip2mm

\ $(i)$ If $w_1 > w_2$ then condition $(C_1)$ holds;

\vskip.5mm

$(ii)$ If $w_1 = w_2$ but $K_{\Gamma_1} \not\subset K_{\Gamma_2}$
then again $(C_1)$ hods.

\vskip2mm

\noindent Thus, unless $w_1 = w_2$ and $K_{\Gamma_1} =
K_{\Gamma_2}$, condition $(C_i)$ holds for at least one $i \in \{1 ,
2\}$.
\end{thm}

In \cite{PR-Fields} we prove much more precise results in the case
where the $\Gamma_i$ are arithmetic, which leads to the geometric
applications described above.  We refer the interested reader to
\cite{PR-Fields} for the technical formulations of these results;
a point we would like to make here, however, is that our
``algebraic" results (i.e., those asserting that condition $(C_i)$
holds in certain situation) do not depend on Schanuel's
conjecture.

\vskip3mm

\section{Generic elements and tori}\label{S:Gen}

The analysis of weak commensurability and its variations in
\cite{GarR},  \cite{PR-IHES} and  \cite{PR-Fields} relies on the
remarkable fact, first established in \cite{PR-Reg}, that any
Zariski-dense subgroup of the group of rational points of a
semi-simple group over a finitely generated field of characteristic
zero contains special elements, to be called {\it generic elements} here. It is
convenient to begin our discussion of these elements with the
definition of {\it generic tori}.

Let $G$ be a connected semi-simple algebraic group defined over an
infinite field $K$. Fix a maximal $K$-torus $T$ of $G$, and, as
usual, let $\Phi = \Phi(G , T)$ denote the corresponding root
system, and let $W(G , T)$ be its Weyl group. Furthermore, we let
$K_T$ denote the (minimal) splitting field of $T$ in a fixed
separable closure $\overline{K}$ of $K$. Then the natural action of
the Galois group $\Ga(K_T/K)$ on the character group $X(T)$ of $T$ induces
an injective homomorphism
$$
\theta_T \colon \Ga(K_T/K) \to \mathrm{Aut}(\Phi(G , T)).
$$
We say that $T$ is {\it generic} (over $K$) if
\begin{equation}\label{E:Gen1}
\theta_T(\Ga(K_T/K)) \supset W(G , T).
\end{equation}
For example, any maximal $K$-torus of $G = \mathrm{SL}_n/K$ is of
the form $T = \mathrm{R}^{(1)}_{E/K}(\mathrm{GL}_1)$ for some
$n$-dimensional commutative  \'etale $K$-algebra $E$. Then such a torus is
generic over $K$ if and only if $E$ is a separable field extension
of $K$ and the Galois group of the normal closure $L$ of $E$ over
$K$ is isomorphic to the symmetric group $S_n$. It is well-known
that for each $n \geqslant 2$ one can write down a system of congruences
such that any monic polynomial $f(t) \in \Z[t]$ satisfying this
system of congruences has Galois group $S_n$. It turns out that one
can prove a similar statement for maximal tori of an arbitrary
semi-simple algebraic group $G$ over a finitely generated field $K$
of characteristic zero (cf.\,\cite[Theorem 3]{PR-Reg}). In order to
avoid technical details, we will restrict ourselves here to the case
of absolutely almost simple groups.

\begin{thm}\label{T:Gen1}
{\rm (\cite[Theorem 3.1]{PR-IHES})} Let $G$ be a connected
absolutely almost simple algebraic group over a finitely generated
field $K$ of characteristic zero, and let $r$ be the number of
nontrivial conjugacy classes of the Weyl group of $G$. Then

\vskip2mm

\noindent {\rm (1)} \parbox[t]{13.5cm}{There exist $r$ inequivalent
nontrivial discrete valuations $v_1, \ldots , v_r$ of $K$ such that
the completion $K_{v_i}$ is locally compact and $G$ splits over
$K_{v_i}$ for all $i = 1, \ldots , r$;}

\vskip1mm

\noindent {\rm (2)} \parbox[t]{13.5cm}{For any choice of discrete
valuations $v_1, \ldots , v_r$ as in {\rm (1)}, one can find maximal
$K_{v_i}$-tori $T(v_i)$ of $G$, one for each $i \in \{1, \ldots ,
r\}$, with the property that  any maximal $K$-torus $T$ of $G$ which
is conjugate to $T(v_i)$ by an element of $G(K_{v_i})$, for all $i =
1, \ldots , r$, is generic (i.e., the inclusion (\ref{E:Gen1})
holds).}
\end{thm}

The first assertion is an immediate consequence of the following,
which actually shows that we can  find the $v_j$'s so that $K_{v_j}
= \Q_{p_j}$, where $p_1, \ldots ,p_r$ are distinct primes.

\begin{prop}\label{P:Gen1}
{\rm (\cite{PR-Sub}, \cite{PR-Reg})} Let $\mathcal{K}$ be a finitely
generated field of characteristic zero and $\mathcal{R} \subset
\mathcal{K}$ a finitely generated subring. Then there exists an
infinite set $\Pi$ of primes such that for each $p \in \Pi$ there
exists an embedding $\varepsilon_p \colon \mathcal{K}
\hookrightarrow \Q_p$ with the property $\varepsilon_p(\mathcal{R})
\subset \Z_p$.
\end{prop}

\vskip2mm

To sketch a proof of the second assertion of Theorem \ref{T:Gen1},
we fix a maximal $K$-torus $T_0$ of $G$. Given any other maximal
torus $T$ of $G$ defined over an extension $F$ of $K$ there exists
$g \in G(\overline{F})$ such that $T = \iota_g(T_0)$, where
$\iota_g(x) = gxg^{-1}$. Then $\iota_g$ induces an isomorphism
between the Weyl groups $W(G , T_0)$ and $W(G , T)$. A different
choice of $g$ will change this isomorphism by an inner automorphism
of the Weyl group, implying that there is a {\it canonical
bijection} between the sets  $[W(G , T_0)]$ and $[W(G , T)]$ of
conjugacy classes in the respective groups;  we will denote this
bijection by $\iota_{T_0, T}$.

Now, let $v$ be a nontrivial discrete valuation of $K$ such that the
completion $K_v$ is locally compact and splits $T_0$. Using
the Frobenius automorphism of the maximal
unramified extension $K^{\mathrm{ur}}_v$ in conjunction with the
fact that $H^1(K_v , \widetilde{G})$, where $\widetilde{G}$ is the
simply connected cover of $G$, vanishes (cf.\,\cite{BrT},
\cite{Kneser}),\footnote{One can alternatively use the fact that if
we endow $\widetilde{G}$ with the structure of a group scheme over
$\mathcal{O}_v$ as a Chevalley group, then
$H^1(K^{\mathrm{ur}}_v/K_v ,
\widetilde{G}(\mathcal{O}^{\mathrm{ur}}_v))$, where
$K^{\mathrm{ur}}_v$ is the maximal unramified extension of $K_v$
with the valuation ring $\mathcal{O}^{\mathrm{ur}}_v$, vanishes,
which follows from Lang's theorem \cite{Lang} or its generalization
due to Steinberg  \cite{Stein}, see \cite[Theorem 6.8]{PlR}.} one
shows that given a nontrivial conjugacy class $c \in [W(G, T_0)]$,
one can find a maximal $K_v$-torus $T(v , c)$ such that given any
maximal $K_v$-torus $T$ of $G$ that is conjugate to $T(v , c)$ by an
element of $G(K_v)$, for its splitting field ${K_v}_T$ we have
\begin{equation}
\theta_T(\Ga({K_v}_T/K_v)) \cap \iota_{T_0, T}(c) \neq \emptyset.
\end{equation}
Now, if $v_1, \ldots , v_r$ are as in part (1), then using the weak
approximation property of  the variety of maximal tori (cf.\,\cite{PlR},
Corollary 7.3), one can pick a maximal $K$-torus $T_0$ which
splits over $K_{v_i}$ for all $i = 1, \ldots , r$. Let $c_1, \ldots
, c_r$ be the nontrivial conjugacy classes of $W(G , T_0)$. Set $T(v_i)
= T(v_i , c_i)$ for $i = 1, \ldots , r$ in the above notation. Then
it is not difficult to show that the tori $T(v_1), \ldots , T(v_r)$
are as required.

The method described above enables one to construct generic tori
with various additional properties, in particular, having prescribed
local behavior.
\begin{cor}\label{C:Gen1}
{\rm (Corollary 3.2 in \cite{PR-IHES})}. Let $G$
and $K$ be as in Theorem \ref{T:Gen1}, and let $V$ be a finite set
of inequivalent nontrivial rank 1 valuations of $K$. Suppose that
for each $v \in V$ we are given a maximal $K_v$-torus $T(v)$ of $G$.
Then there exists a maximal $K$-torus $T$ of $G$ for which
(\ref{E:Gen1}) holds and which is conjugate to $T(v)$ by an element
of $G(K_v)$, for all $v \in V$.
\end{cor}

\vskip2mm

It should be noted that the method of $p$-adic embeddings that we
used in the proof of Theorem \ref{T:Gen1}, and which is based on
Proposition \ref{P:Gen1}, has many other applications -- see
\cite{PR-Chinese}.

\vskip2mm

We are now prepared to discuss {\it generic elements} whose
existence in any finitely generated Zariski-dense subgroup is the
core issue in this section.

\vskip2mm

\noindent {\bf 9.4. Definition.} Let $G$ be a connected semi-simple
algebraic group defined over a field $K$. A regular semi-simple
element $g \in G(K)$ is called {\it generic} (over $K$) if the maximal  torus
$T := Z_G(g)^{\circ}$ is generic (over $K$). We shall refer to $T$ as the torus associated with $g$.

\vskip1mm

Before  proceeding with the discussion of generic elements, we would
like to point out that some authors adopt a slight variant of this
definition by requiring that the extension $K_g$ of $K$ generated by
the eigenvalues of $g$ be ``generic," which is more consistent with
the notion of a ``generic polynomial" in Galois theory. We note that
$K_g \subset K_T$ making the Galois group $\Ga(K_g/K)$ a {\it
quotient} of the group $\Ga(K_T/K)$. Then the requirement that the
order of $\Ga(K_g/K)$ be divisible by $\vert W(G , T) \vert$ (which
is probably one of the most natural ways to express the
``genericity" of $K_g/K$) a priori may not imply the inclusion
(\ref{E:Gen1}), which is most commonly used in applications
(although the former does imply the latter if $G$ is an inner form of a split group\,--\:cf.\,\cite[Lemma 4.1]{PR-IHES}, in particular, if  $G$ is absolutely 
almost simple of type different from $\textsf{A}_{n}$ $(n
> 1)$, $\textsf{D}_{n}$ $(n \geqslant 4)$ and $\textsf{E}_6$).
On the other hand, even for a regular element $g \in T(K)$, where
$T$ is a maximal generic $K$-torus of $G$, the field $K_g$ may be
strictly smaller than $K_T$. (Example: Let $G = \mathrm{PSL}_2$ over
$K$, and let $T$ be a maximal $K$-torus of $G$ of the form
$\mathrm{R}^{(1)}_{L/K}(\mathrm{GL}_1)$ where $L/K$ is a quadratic
extension; then an element $g \in T(K)$ of order two is regular with
$K_g = K$, while $K_T = L$). This problem, however, does not arise
if $G$ is absolutely almost simple, $T$ is generic and $g \in T(K)$ has infinite order. 
Indeed, then the $K$-torus $T$ is {\it irreducible}, i.e., it does not contain any 
proper $K$-subtori since $W(G,T)$ acts irreducibly on $X(T)\otimes_{\Z}\Q$. 
It follows that every element $g \in T(K)$ of
infinite order generates a Zariski-dense subgroup of $T$, hence $K_g
= K_T$, so the order of $\Ga(K_g/K)$ is divisible by
$\vert W(G , T) \vert$. Thus, the above variant of the definition of
a generic element leads to the same concept for elements of infinite
order. Recall that according to a famous result due to Selberg
(cf.\,\cite[Theorem 6.11]{Rag-book}), any finitely generated
subgroup $\Gamma$ of  $G(K)$ contains a torsion-free subgroup
$\Gamma'$ of finite index, which therefore is also Zariski-dense.
So, the following theorem  (Theorem 9.6) that asserts the existence of generic
elements in an arbitrary Zariski-dense subgroup in the sense of our
definition also implies the existence of generic elements in the sense of the other definition. 
\vskip2mm

\noindent{\bf 9.5.}  Let $K$ be a field and $G$ a connected absolutely almost simple $K$-group. Let $g\in G(K)$ be a generic element and let $T = Z_G(g)^{\circ}$. Then $g\in T(K)$ (see \cite{Borel}, Corollary 11.12).  
As $T$ does not contain proper $K$-subtori, the cyclic group generated by any $t\, (\in T(K))$ of infinite order is Zariski-dense in $T$.  So $Z_G(t)=Z_G(T)=T$; which implies that $t$ is generic (over $K$) and 
$Z_G(t)$ is connected; in particular, if $g$ is of infinite order, then $Z_G(g) \,(=T)$ is connected.  If $n\in G(K)$ is such that $ntn^{-1}$ commutes with $t$, then $n$ normalizes $T$, i.e., $n$ lies in the normalizer 
$N_G(T)(K)$ of $T$ in $G(K)$. If $x$ is an element of $G(K)$ of infinite order such that for some nonzero integer $a$, $t :=x^a$ lies in $T(K)$, then 
as $x$ commutes with $t$, it commutes with $T$ and hence it lies in $T(K)$. 
\vskip1mm
It is also clear that the tori associated to two generic elements are equal if and only if the elements commute.
\vskip1mm

It is easily seen that the natural action of  $N_G(T)(K)$ on the character group $X(T)$  commutes  
with the natural action of \,${\mathrm{Gal}}(K_T/K)$. Now since  $T$ is generic, $\theta_T({\mathrm{Gal}}(K_T/K))$ $\supset W(G,T)$, and as  $W(G,T)$ 
acts irreducibly on $X(T)\otimes_{\Z}\C$, we see that the elements of $N_G(T)(K)$ act by $\pm I$ on $X(T)$. Therefore,  for $n\in N_G(T)(K)$, $n^2$ commutes 
with $T$ and hence lies in $T(K)$. Now if $n\in G(K)$ is such that $ngn^{-1}$ commutes with $g$, then $n$ belongs to $N_G(T)(K)$ and $n^2\in T(K)$.   
So, if moreover $n$ is of infinite order, then it actually lies in $T(K)$.  Thus an element of $G(K)$ 
of infinite order which does not belong to $T(K)$ cannot normalize $T$. 

\addtocounter{thm}{2}

\begin{thm}\label{T:Gen2}
Let $G$ be a connected absolutely almost simple algebraic group over
a finitely generated field $K$, and let $\Gamma$ be a
finitely generated Zariski-dense subgroup of $G(K)$. Then $\Gamma$ contains a
generic element (over $K$) of infinite order.
\end{thm}

(It is not difficult to show, e.g.\:using Burnside's characterization
of absolutely irreducible linear groups, that any Zariski-dense
subgroup $\Gamma \subset G(K)$ contains a {\it finitely generated}
Zariski-dense subgroup, so the assumption in the theorem that
$\Gamma$ be finitely generated can actually be omitted.)

\vskip2mm

\noindent {\it Sketch of the proof.} Fix a matrix $K$-realization $G
\subset \mathrm{GL}_N$, and pick a finitely generated subring $R
\subset K$ so that $\Gamma \subset G(R) := G \cap \mathrm{GL}_N(R)$.
Let $r$ be the number of nontrivial conjugacy classes in the Weyl
group $W(G , T)$. Using Proposition \ref{P:Gen1}, we can find $r$
distinct primes $p_1, \ldots , p_r$ such that for each $i \leqslant
r$ there exists an embedding $\varepsilon_i \colon K \hookrightarrow
\Q_{p_i}$  such that $\varepsilon_i(R) \subset \Z_{p_i}$ and $G$
splits over $\Q_{p_i}$. Let $v_i$ be the discrete valuation of $K$
obtained as the pullback of the $p_i$-adic valuation of $\Q_{p_i}$
so that $K_{v_i} = \Q_{p_i}$. Pick maximal $K_{v_i}$-tori $T(v_1),
\ldots , T(v_r)$ as in part (2) of Theorem \ref{T:Gen1}. Let
$\Sigma_i$ be the Zariski-open $K_{v_i}$-subvariety  of regular
elements in $T(v_i)$. It follows from the Implicit Function Theorem
that the image $\Omega_i$ of the map $$G(K_{v_i}) \times
\Sigma_i(K_{v_i}) \to G(K_{v_i}), \ \ (g , t) \mapsto gtg^{-1},$$ is
open in $G(K_{v_i})$ and intersects every open subgroup of the
latter. On the other hand, as explained in \cite[\S 3]{Rap-SA},
the closure of the image of the diagonal embedding $\Gamma
\hookrightarrow \prod_{i = 1}^r G(K_{v_i})$ is open, hence contains
some $U = \prod_{i = 1}^r U_i$ where $U_i \subset G(K_{v_i})$ is a
Zariski-open torsion-free subgroup. Then $$U_0 := \prod_{i = 1}^r
(U_i \cap \Omega_i)$$ is an open set that intersects $\Gamma$, and
it follows from our construction that any element $g \in \Gamma \cap
U_0$ is a generic element of infinite order.
\vskip2mm

Basically, our proof  shows that given a finitely generated
Zariski-dense subgroup $\Gamma$ of  $G(K)$, one can produce a
finite system of congruences (defined in terms of suitable
valuations of $K$) such that the set of elements $\gamma \in \Gamma$
satisfying this system of congruences consists entirely of generic
elements (and additionally this set is in fact a coset of a finite
index subgroup in $\Gamma$, in particular, it is Zariski-dense in
$G$). Recently, Jouve-Kowalski-Zywina \cite{JKZ}, Gorodnik-Nevo
\cite{GorNe} and Lubotzky-Rosenzweig \cite{LubRo} developed
different {\it quantitative} ways of showing that generic elements
exist in abundance (in fact, these results demonstrate  that ``most"
elements in $\Gamma$ are generic). More precisely, the result of
\cite{GorNe} gives the asymptotics of the number of generic elements
of a given height in an arithmetic group, while the results of
\cite{LubRo}, generalizing earlier results of \cite{JKZ},  are
formulated in terms of random walks on groups and apply to arbitrary
Zariski-dense subgroups in not necessarily connected semi-simple
groups. These papers introduce several new ideas and techniques, but
at the same time employ the elements of the argument from
\cite{PR-Reg} we outlined above.

\vskip2mm

The proofs of the results in \cite{PR-IHES}, \cite{PR-Fields} use
not only Theorem \ref{T:Gen2} itself but also its different
variants that provide generic elements with additional properties,
e.g. having  prescribed local behavior (cf.\,Corollary \ref{C:Gen1}).
We refer the interested reader to these papers for precise
formulations (which are rather technical), and will only indicate
here the basic ``multidimensional" version of Theorem \ref{T:Gen2}
that was developed in \cite{PR-Fields}.
\begin{thm}\label{T:Gen3}
{\rm (cf.\,\cite[Theorem 3.4]{PR-Fields})} Let $G$, $K$  and $\Gamma
\subset G(K)$ be as in Theorem \ref{T:Gen2}. Then for any $m
\geqslant 1$ one can find generic semisimple elements $\gamma_1,
\ldots , \gamma_m \in \Gamma$ of infinite order that are
multiplicatively independent.
\end{thm}

Finally, we would like to formulate a result that enables one to
pass from the weak commensurability of two generic semi-simple
elements to an isogeny, and in most cases even to an isomorphism, of
the ambient tori. This result relates the analysis of weak
commensurability to the problem of characterizing algebraic group
having the same isomorphism/isogeny classes of maximal tori.
\begin{thm}\label{T:Gen4}
{\rm (Isogeny Theorem, \cite[Theorem 4.2]{PR-IHES})} Let $G_1$ and
$G_2$ be two connected absolutely almost simple algebraic groups
defined over an infinite field $K$, and let $L_i$ be the minimal
Galois extension of $K$ over which $G_i$ becomes an inner form of a
split group. Suppose that for $i = 1, 2$, we are given a semi-simple
element $\gamma_i \in G_i(K)$ contained in a maximal $K$-torus $T_i$
of $G_i$. Assume that (i) $G_1$ and $G_2$ are either of the same
Killing-Cartan type, or one of them is of type $\textsf{B}_{n}$ and
the other is of type $\textsf{C}_{n}$, (ii) $\gamma_1$ has infinite
order, (iii) $T_1$ is $K$-irreducible, and (iv) $\gamma_1$ and
$\gamma_2$ are weakly commensurable. Then

\vskip2mm

\noindent {\rm (1)} \parbox[t]{13.5cm}{there exists a $K$-isogeny
$\pi \colon T_2 \to T_1$ which carries $\gamma^{m_2}_2$ to
$\gamma^{m_1}_1$ for some integers $m_1 , m_2 \geqslant 1$;}

\vskip1mm

\noindent {\rm (2)} \parbox[t]{13.5cm}{if $L_1 = L_2 =: L$ and
$\theta_{T_1}(\Ga(L_{T_1}/L)) \supset W(G_1 , T_1)$, then $\pi^*
\colon X(T_1)\otimes_{\Z}\Q \to X(T_2) \otimes_{\Z} \Q$ has the
property that $\pi^*(\Q \cdot \Phi(G_1 , T_1)) = \Q \cdot \Phi(G_2 ,
T_2)$. Moreover, if $G_1$ and $G_2$ are of the same Killing-Cartan
type different from $\textsf{B}_2 = \textsf{C}_2$, $\textsf{F}_4$ or
$\textsf{G}_2$, then a suitable rational multiple of $\pi^*$  maps
$\Phi(G_1 , T_1)$ onto $\Phi(G_2 , T_2)$, and if $G_1$ is of type
$\textsf{B}_{n}$ and $G_2$ is of type $\textsf{C}_{n}$, with $n >
2$, then a suitable rational multiple $\lambda$ of $\pi^*$ takes the
long roots in $\Phi(G_1 , T_1)$ to the short roots in $\Phi(G_2 ,
T_2)$ while $2\lambda$ takes the short roots in $\Phi(G_1 , T_1)$ to
the long roots in $\Phi(G_2 , T_2)$.}
\end{thm}

It follows that in the situations where $\pi^*$ can be, and has
been, scaled so that $\pi^*(\Phi(G_1 , T_1)) = \Phi(G_2 , T_2)$, it
induces $K$-isomorphisms $\widetilde{\pi} \colon \widetilde{T}_2 \to
\widetilde{T}_1$ and $\overline{\pi} \colon \overline{T}_2 \to \overline{T}_1$ between
the corresponding tori in the simply connected and adjoint groups
$\widetilde{G}_i$ and $\overline{G}_i$, respectively, that extend to
$\overline{K}$-isomorphisms $\widetilde{G}_2 \to \widetilde{G}_1$ and
$\overline{G}_2 \to \overline{G}_1$. Thus, the fact that Zariski-dense
torsion-free subgroups $\Gamma_1 \subset G_1(K)$ and $\Gamma_2
\subset G_2(K)$ are weakly commensurable implies (under some minor
technical assumptions) that $G_1$ and $G_2$ have the same
$K$-isogeny classes (and under some additional assumptions, even the
same $K$-isomorphism classes) of generic maximal  $K$-tori that
nontrivially intersect $\Gamma_1$ and $\Gamma_2$, respectively.

\vskip2mm

For a ``multidimensional" version of Theorem \ref{T:Gen4}, which is
formulated using the notion of weak containment (see \S
\ref{S:Fields}) in place of weak commensurability, see
\cite[Theorem 2.3]{PR-Fields}.
\vskip3mm

{\bf 9.9.} We would like to conclude this section with one new
observation (Theorem 9.10) which is directly related to the main theme of the
workshop -- thin groups. This observation was inspired by the conversations of  the first-named author with Igor Rivin at the Institute for Advanced Study.

Let $G$ be a connected absolutely almost simple
algebraic group over a field $K$ of characteristic zero, and let $T$ be a maximal $K$-torus of $G$. We
let $\Phi_{>}(G , T)$ (resp., $\Phi_{<}(G ,
T)$) denote the set of all long (resp., short) roots in the root
system $\Phi(G , T)$; by convention,
$$
\Phi_{>}(G , T) = \Phi_{<}(G , T) = \Phi(G
, T)$$
if all roots have the same length. Furthermore, we let
$G^{>}_T$ denote the $K$-subgroup of $G$ generated by $T$
and the 1-parameter unipotent subgroups $U_{a}$ for $a \in
\Phi_{>}(G , T)$. Then $G^{>}_T$ is a
connected semi-simple subgroup of $G$  of maximal absolute rank (so, in fact, just the
$U_{a}$'s for $a \in \Phi_{>}(G , T)$ generate $G^{>}_T$). By direct inspection, one verifies
that $G^{>}_T \neq G$ precisely when $\Phi(G , T)$ has
roots of different lengths, and then $G^{>}_T$ is a
semi-simple group of type $(\textsf{A}_1)^n$ if $G$ is of type
$\textsf{C}_n$, and an absolutely almost simple group of type
$\textsf{D}_n$, $\textsf{D}_4$ and $\textsf{A}_2$ if $G$ is of type
$\textsf{B}_n$, $\textsf{F}_4$ and $\textsf{G}_2$, respectively. 
On the other hand, the subgroups $U_a$ for $a \in
\Phi_{<}(G , T)$ generate $G$ in all cases. Finally,
for any connected subgroup of $G$ containing $T$ there exists a
subset $\Psi \subset \Phi(G , T)$ such that $G$ is generated by $T$
and $U_a$ for all $a \in \Psi$. \addtocounter{thm}{1}
\begin{thm}\label{T:thin}
Let $g$ be a generic element of infinite order and  $T:=Z_G(g)$ be the associated maximal torus. 
Let $x\in G(K)$ be any element of infinite order not contained in $T(K)$. 
Furthermore, let $\Gamma$ be the (abstract) subgroup of $G(K)$
generated by $g$ and $x$, and let $H$ be the identity component of the Zariski-closure of $\Gamma$. 
Then either $H = G$ or $H
= G^{>}_T$. Consequently,
$g$ and $x$ generate a Zariski-dense subgroup
of $G$ if all roots in the root system $\Phi(G,T)$ are of same length.
\end{thm}
\begin{proof}
As $g$ is a generic element of infinite order, the cyclic group generated by it is Zariski-dense in $T$, and so the cyclic group generated by $xgx^{-1}$ 
is Zariski-dense in the conjugate torus $xTx^{-1}$. Since $x\notin T(K)$ and is of infinite order, it cannot normalize $T$ (see 9.5). Thus $H$ contains at least  two different (generic) 
maximal $K$-tori, namely $T$ and $xTx^{-1}$.  Assume that $H\ne G $. Since $H$ is connected and properly contains
$T$, it must contain a 1-parameter subgroup $U_a$ for some
$a \in \Phi(G , T)$. Then being defined over $K$, $H$ also
contains $U_b$ for all $b$ of the form $b = \sigma(a)$
with $\sigma \in \Ga(K_{T}/K)$. Now since $T$ is
generic, using the fact that the Weyl group $W(G , T)$ acts
transitively on the subsets of roots of same length (cf.\,\cite[Ch.\,VI, Prop.\,11]{Bou}), 
we see that $H$ contains $U_b$ for all roots $b 
\in \Phi(G , T)$ of same length as $a$. If $a$
were a short root then the above remarks would imply that $H = G$,
which is not the case. Thus, $a$ must be long, and therefore
$H$ contains $G^{>}_T$ but does not contain
$U_b$ for any short root $b$. This clearly implies
that $H = G^{>}_T$.
\end{proof}

\vskip2mm

\noindent{\bf Remark 9.11.} It is worth noting that the types with roots of different lengths
are honest exceptions in Theorem \ref{T:thin} in the sense that for
any absolutely almost simple algebraic group $G$ of one of those
types over a finitely generated field $K$ one can find 
two generic elements $\gamma_1 , \gamma_2 \in G(K)$ that generate
$G^{>}_T\neq G$ for a generic maximal $K$-torus $T$. 
To see this, we first pick an arbitrary generic element $\gamma_1
\in G(K)$ of infinite order provided by Theorem \ref{T:Gen2}, and
let $T = Z_G(\gamma_1)$ be the corresponding torus. Since $H :=
G^{>}_T$ is semi-simple, the group $H(K)$ is
Zariski-dense in $H$ (cf.\,\cite[Corollary 18.3]{Borel}). So, there
exists $h \in H(K)$ such that $\gamma_2 := h\gamma_1 h^{-1} \notin
T(K)$. Then $\gamma_2 \in H(K)$ is also generic over $K$, and the
Zariski-closure of the subgroup generated by $\gamma_1$ and
$\gamma_2$ is contained in (in fact, is equal to) $H$.

\section{Some open problems}\label{S:Prob}

The analysis of weak commensurability has led to a number of
interesting problems in the theory of algebraic and Lie groups (cf.,
for example, \S \ref{S:Tori}) and its applications to locally
symmetric spaces, and we would like to conclude this article with a
brief discussion of some of these problems.

According to Theorem \ref{T:WC7}, if two lattices in the
groups of rational points of connected absolutely almost simple groups over a
nondiscrete locally compact field are weakly commensurable and one
of the lattices is arithmetic, then so is the other. At the same time, it has been shown by
means of an example
(cf. \cite[Remark 5.5]{PR-IHES}) that a Zariski-dense subgroup
weakly commensurable to a {\it rank one} arithmetic subgroup need
not be arithmetic. It would be interesting, however, to understand
what happens with  {\it higher rank} $S$-arithmetic subgroups.

\vskip2mm

\noindent {\bf Problem \ref{S:Prob}.1.} {\it Let $G_1$ and $G_2$ be
two connected absolutely almost simple algebraic groups defined over
a field $F$ of characteristic zero, and let $\Gamma_1 \subset
G_1(F)$ be a Zariski-dense $(K , S)$-arithmetic subgroup whose
$S$-rank\footnotemark is $\geqslant 2$. If $\Gamma_2 \subset G_2(F)$
is a Zariski-dense subgroup weakly commensurable to $\Gamma_1$, then
is $\Gamma_2$ necessarily $S$-arithmetic?}

\footnotetext{We recall that if $\Gamma$ is $(\sG, K,
S)$-arithmetic, then the $S$-rank of $\Gamma$ is defined to be
$ \sum_{v \in S} \mathrm{rk}_{K_v}\: \sG$, where
$\mathrm{rk}_F \: \sG$ denotes the rank of $\sG$ over a field $F
\supset K$.}

\vskip2mm  This problem appears to be very challenging; the answer
is not known even in the cases where $\Gamma_1$ is $\mathrm{SL}_3(\Z)$ or
$\mathrm{SL}_2(\Z[1/p])$. One should probably start by considering Problem
\ref{S:Prob}.1 in a more specialized situation, e.g., assuming that
$F$ is a nondiscrete locally compact field, $\Gamma_1 \subset
G_1(F)$ is a {\it discrete} Zariski-dense (higher rank)
$S$-arithmetic subgroup, and $\Gamma_2 \subset G_2(F)$ is a
(finitely generated) {\it discrete} Zariski-dense subgroup weakly
commensurable to $\Gamma_1$ (these restrictions would eliminate
$\mathrm{SL}_2(\Z[1/p])$ as a possibility for $\Gamma_1$, but many
interesting groups such as $\mathrm{SL}_3(\Z)$ would still be included). The
nature of these assumptions brings up another question of
independent interest.

\vskip2mm

\noindent {\bf Problem \ref{S:Prob}.2.} {\it Let $G_1$ and $G_2$ be
connected absolutely almost simple algebraic groups over a
nondiscrete locally compact field $F$, and let $\Gamma_i$ be a
finitely generated Zariski-dense subgroup of $G_i(F)$ for $i = 1,
2$. Assume that $\Gamma_1$ and $\Gamma_2$ are weakly commensurable.
Does the discreteness of $\Gamma_1$ imply the discreteness of
$\Gamma_2$?}

\vskip2mm

An affirmative answer to Problem \ref{S:Prob}.2 was given in
\cite[Propisition 5.6]{PR-IHES} for the case where $F$ is a {\it
nonarchimedean} local field, but the case $F = \R$ or $\C$ remains
open. Another interesting question is whether weak commensurability
preserves cocompactness of lattices.

\vskip2mm

\noindent {\bf Problem \ref{S:Prob}.3.} {\it Let $G_1$ and $G_2$ be
connected absolutely almost simple algebraic groups over $F = \R$ or
$\C$, and let $\Gamma_i \subset G_i(F)$ be a lattice for $i = 1, 2$.
Assume that $\Gamma_1$ and $\Gamma_2$ are weakly commensurable. Does
the compactness of $G_1(F)/\Gamma_1$ imply the compactness of
$G_2(F)/\Gamma_2$?\footnotemark}

\footnotetext{It is well-known that for a semi-simple algebraic
group $G$ over a nondiscrete nonarchimedean locally compact field
$F$ of characteristic zero and a discrete subgroup $\Gamma \subset
G(F)$, the quotient $G(F)/\Gamma$ has finite measure if and only if
it is compact, so the problem in this case becomes vacuous.}

\vskip2mm

We recall that the co-compactness of a lattice in a semi-simple real
Lie group is equivalent to the absence of nontrivial unipotents in
it (cf.\,\cite[Corollary 11.13]{Rag-book}). So, Problem
\ref{S:Prob}.3 can be rephrased as the question whether for two
weakly commensurable lattices $\Gamma_1$ and $\Gamma_2$, the
existence of nontrivial unipotent elements in one of them implies
their existence in the other; in this form the question is
meaningful for arbitrary Zariski-dense subgroups (not necessarily
discrete or of finite covolume). The combination of Theorems
\ref{T:WC6} and \ref{T:WC7} implies the affirmative answer to
Problem \ref{S:Prob}.3 in the case where one of the lattices is
arithmetic, but no other cases have been considered so far.

\vskip2mm

From the general perspective, one important problem is to try to
generalize our results on length-commensurable and/or isospectral
arithmetically defined locally symmetric spaces of absolutely simple
real Lie groups to arithmetically defined locally symmetric spaces
of arbitrary semi-simple Lie groups, or at least those of
$\R$-simple Lie groups. To highlight the difficulty, we will make
some comments about the latter case. An $\R$-simple adjoint group
$G$ can be written in the form $G = \mathrm{R}_{\C/\R}(H)$
(restriction of scalars) where $H$ is an absolutely simple complex
algebraic group. Arithmetic lattices in $G(\R) \simeq H(\C)$ come
from the forms of $H$ over a number field admitting exactly one
complex embedding. The analysis of weak commensurability of even
arithmetic lattices $\Gamma_1 \subset G_1(\R)$ and $\Gamma_2 \subset
G_2(\R)$, where $G_i = \mathrm{R}_{\C/\R}(H_i)$ for $i = 1, 2$,
cannot be implemented via the study of the forms of the $G_i$'s,
forcing us to study directly the forms of the $H_i$'s. But the
relation of weak commensurability of semi-simple elements $\gamma_1
\in \Gamma_1$ and $\gamma_2 \in \Gamma_2$ in terms of $G_1$ and
$G_2$, i.e. the fact that $\chi_1(\gamma_1) = \chi_2(\gamma_2) \neq
1$ for some characters $\chi_i$ of maximal $\R$-tori $T_i$ of $G_i$
such that $\gamma_i \in T_i(\R)$, translates into a significantly
more complicated relation in terms of $H_1$ and $H_2$. Indeed, pick
maximal $\C$-tori $S_i$ of $H_i$ so that $T_i =
\mathrm{R}_{\C/\R}(S_i)$, and let $\delta_i \in S_i(\C)$ be the
element corresponding to $\gamma_i$ under the identification
$T_i(\R) \simeq S_i(\C)$. Then there exist characters $\chi'_i ,
\chi''_i$ of $S_i$ such that $\chi_i(\gamma_i) = \chi'_i(\delta_i)
\overline{\chi''_i(\delta_i)}$. So, the relation of weak
commensurability of $\gamma_1$ and $\gamma_2$ assumes the following
form in terms of $\delta_1$ and $\delta_2$:
$$
\chi'_1(\delta_1) \overline{\chi''_1(\delta_1)} = \chi'_2(\delta_2)
\overline{\chi''_2(\delta_2)}.
$$
It is not  clear  if this type of relation would lead to
the results similar to those we described in this article for the
weakly commensurable arithmetic subgroups of absolutely almost
simple groups. So, the general problem at this stage is to formulate
for general semi-simple groups (or at least $\R$-simple groups) the
``right" notion of weak commensurability and explore its
consequences. We will now formulate a particular case of this
general program that would be interesting for geometric
applications.

\vskip2mm

\noindent {\bf Problem \ref{S:Prob}.4.} {\it Let $G_1$ and $G_2$ be
almost simple complex algebraic groups. Two semisimple elements
$\gamma_i \in G_i(\C)$ are called $\R$-weakly commensurable if there
exist complex maximal tori $T_i$ of $G_i$ for $i = 1, 2$ such that
$\gamma_i \in T_i(\C)$ and for suitable characters $\chi_i$ of $T_i$
we have
$$
\vert \chi_1(\gamma_1) \vert = \vert \chi_2(\gamma_2) \vert \neq 1.
$$
Furthermore, Zariski-dense (discrete) subgroups $\Gamma_i \subset
G_i(\C)$ are $\R$-weakly commensurable if every semisimple element
$\gamma_1 \in \Gamma_1$ of infinite order is $\R$-weakly
commensurable to some semisimple element $\gamma_2 \in \Gamma_2$ of
infinite order, and vice versa. Under what conditions does the
$\R$-weak commensurability of Zariski-dense (arithmetic) lattices
$\Gamma_i \subset G_i(\C)$ $(i = 1, 2)$ imply their
commensurability?}

\vskip2mm

The result of \cite{CHLR} seems to imply that the $\R$-weak
commensurability of arithmetic lattices in $\mathrm{SL}_2(\C)$ does
imply their commensurability, but no other results in this direction
are available.

\vskip3mm

Turning now to the geometric aspect, we would like to reiterate that
most of our results deal with the analysis of the new relation of
length-commensurability, which eventually implies the results about
isospectral locally symmetric spaces. At the same time, the general
consequences of isospectrality and iso-length-spectrality are much
better understood than those of length-commensurability. So, as an
overarching problem, we would like to propose the following.

\vskip2mm

\noindent {\bf Problem \ref{S:Prob}.5.} {\it Understand consequences
(qualitative and quantitative) of length-com- mensurability for
locally symmetric spaces.}

\vskip2mm

(Here by {\it quantitative consequences} we mean results stating
that in certain situations a family of length-commensurable locally
symmetric spaces consists either of a single commensurability class
or of a certain bounded number of commensurability classes, and by
{\it qualitative consequences} - results guaranteeing that the
number of commensurability classes in a given class of
length-commensurable locally symmetric spaces is finite.)

\vskip2mm

There are various concrete questions within the framework provided
by Problem \ref{S:Prob}.5 that were resolved in \cite{PR-IHES} for
arithmetically defined locally symmetric spaces but remain open for
locally symmetric spaces which are not arithmetically defined. For example, according to Theorem
\ref{T:WC10}, if $\fX_{\Gamma_1}$ and $\fX_{\Gamma_2}$ are
length-commensurable and at least one of the spaces is
arithmetically defined, then the compactness of one of them implies
the compactness of the other. It is natural to ask if this can be proved for locally symmetric
spaces which are not arithmetically defined.

\vskip2mm

\noindent {\bf Problem \ref{S:Prob}.6.} (Geometric version of
Problem \ref{S:Prob}.3) {\it Let $\fX_{\Gamma_1}$ and
$\fX_{\Gamma_2}$ be length-commensurable locally symmetric spaces of
finite volume. Does the compactness of $\fX_{\Gamma_1}$ always imply
the compactness of $\fX_{\Gamma_2}$?}

\vskip2mm

In \cite[\S9]{PR-IHES}, for each of the exceptional types $\textsf{A}_{n}$, $\textsf{D}_{2n
+ 1}$ $(n > 1)$ and $\textsf{E}_6$, we have constructed  examples of length-commensurable,
but not commensurable,
compact arithmetically defined locally symmetric spaces associated
with a simple real algebraic group of this type. It would be
interesting to see if this construction can be refined to  provide
examples of iso-length spectral or even isospectral compact
arithmetically defined locally symmetric spaces that are not
commensurable.

\vskip2mm

\noindent {\bf Problem \ref{S:Prob}.7.} {\it For each of the types
$\textsf{A}_{n}$ $(n > 1)$, $\textsf{D}_{2n + 1}$ $(n > 1)$ and
$\textsf{E}_6$, construct examples of isospectral compact
arithmetically defined locally symmetric spaces that are not
commensurable.}

\vskip2mm

Currently, such a construction is known only for {\it inner forms}
of type $\textsf{A}_{n}$ (cf.\,\cite{LSV}); it relies on some
delicate results from the theory of automorphic forms \cite{HT}, the
analogues of which are not yet available for groups of other types.

\vskip2mm

As we already mentioned, in \cite{PR-IHES} we focused on the case
where $G_1$ and $G_2$ are absolutely (almost) simple real algebraic
groups. From the geometric perspective, however, it would be
desirable to consider a more general situation where $G_1$ and $G_2$
are allowed to be either arbitrary real semi-simple groups (without
compact factors), or at least arbitrary $\R$-simple groups. This
problem is intimately related to the problem, discussed above, of
generalizing our results on weak commensurability from absolutely
almost simple to arbitrary semi-simple groups. In particular, a
successful resolution of Problem \ref{S:Prob}.4 would enable us to
extend our results to the (arithmetically defined) locally symmetric
spaces associated with $\R$-simple groups providing thereby a
significant generalization of the result of \cite{CHLR} where the
case $G_1 = G_2 = \mathrm{R}_{\C/\R}(\mathrm{SL}_2)$ (that leads to
arithmetically defined hyperbolic 3-manifolds) was considered.

\vskip2mm

Finally, the proof of the result that connects the
length-commensurability of $\fX_{\Gamma_1}$ and $\fX_{\Gamma_2}$ to
the weak commensurability of $\Gamma_1$ and $\Gamma_2$ relies (at
least in the higher rank case) on Schanuel's conjecture. It would be
interesting to see if our geometric results can be made
independent of Schanuel's conjecture.

\vskip3mm

\noindent {\small {\bf Acknowledgements.} Both authors were
partially supported by the NSF (grants DMS-1001748  and DMS-0965758)
and the Humboldt Foundation. The first-named author thanks the
Institute for Advanced Study (Princeton) for its hospitality and
support during 2012. The hospitality of the MSRI during the workshop
``Hot Topics: Thin Groups and Super-strong Approximation" (February
6-10, 2012) is also thankfully acknowledged. We thank Skip Garibaldi
and Igor Rapinchuk for carefully reading the manuscript and for
their comments and corrections.}

\vskip5mm

\bibliographystyle{amsplain}

\end{document}